[1994/12/01]
\documentclass{ijmart-mod}
\chardef\bslash=`\\ 

\usepackage{graphicx}
\usepackage[breaklinks=true]{hyperref}
\usepackage{mathtools}
\usepackage{array}
\usepackage{multirow}
\usepackage{pdflscape}
\usepackage{afterpage}

\newtheorem{thm}{Theorem}[section]

\newtheorem{lem}[thm]{Lemma}
\newtheorem{prop}[thm]{Proposition}

\theoremstyle{definition}

\theoremstyle{remark}

\newcommand{\eval}[2][\right]{\relax
  \ifx#1\right\relax \left.\fi#2#1\rvert}

\begin{document}
\title{Primitive point packing}

\author[Antoine Deza]{Antoine Deza}
\address{McMaster University, Hamilton, Ontario, Canada}
\email{deza@mcmaster.ca} 

\author[Lionel Pournin]{Lionel Pournin}
\address{Universit{\'e} Paris 13, Villetaneuse, France}
\email{lionel.pournin@univ-paris13.fr}

\begin{abstract}
A point in the $d$-dimensional integer lattice $\mathbb{Z}^d$ is primitive when its coordinates are relatively prime. Two primitive points are multiples of one another when they are opposite, and for this reason, we consider half of the primitive points within the lattice, the ones whose first non-zero coordinate is positive. We solve the packing problem that asks for the largest possible number of such points whose absolute values of any given coordinate sum to at most a fixed integer $k$. We present several consequences of this result at the intersection of geometry, number theory, and combinatorics. In particular, we obtain an explicit expression for the largest possible diameter of a lattice zonotope contained in the hypercube $[0,k]^d$ and, conjecturally of any lattice polytope in that hypercube.
\end{abstract}
\maketitle

\section{Introduction}\label{PPP.sec.0}

Lattice polytopes appear in many branches of mathematics, as for instance in algebraic geometry where they are associated with certain toric varieties. It is noteworthy that, in relation with the study of these and similar objects, methods from combinatorics and algebraic geometry have been beneficial to both fields \cite{AdiprasitoHuhKatz2018,BogartHaaseHeringLorenzNillPaffenholzRoteSantosSchenck2015,DickensteinDiRoccoPiene2009,Huh2012,Pommersheim1993,Stanley1980}. Another branch of mathematics where lattice polytopes turn up is combinatorial optimization. In particular, they encode the feasible domain of a number of optimization problems \cite{AprileCevallosFaenza2018,DelPiaMichini2016,Naddef1989,NillZiegler2011}. In this other context, an important quantity is the largest diameter a polytope can possibly have in terms of a fixed combinatorial or geometric property. Here, by the diameter of a polytope, we mean the diameter of the graph made up of its vertices and edges. This quantity is tightly linked to the complexity of pivoting methods for solving the corresponding optimization problem and it has been extensively studied, for instance as a function of the dimension and number of facets of the polytope \cite{KalaiKleitman1992,KleeWalkup1967,Naddef1989,Sukegawa2019}. However, as the disproval of the Hirsch conjecture \cite{Santos2012} shows, its behavior remains elusive.

Yet another field where lattice polytopes appear, lies at the intersection of geometry and number theory \cite{Andrews1963,BrionVergne1997,Ehrhart1967a,Ehrhart1967b}. For instance, the density of primitive points---points whose coordinates are relatively prime---within the lattice $\mathbb{Z}^d$ is equal to $1/\zeta(d)$, where $\zeta$ denotes Riemann's zeta function \cite{HardyWright1938,KranakisPocchiola1994,Nymann1970}. From that observation, a sharp estimate for the largest possible diameter of a lattice zonotope contained in the hypercube $[0,k]^d$ when $d$ is fixed and $k$ grows large has been derived recently \cite{DezaPourninSukegawa2020}. 
The case of zonotopes---polytopes obtained as a Minkowski sum of line segments---is particularly interesting in relation to the problem of the diameter of polytopes. Indeed, it is conjectured that the largest possible diameter of a lattice polytope contained in the hypercube $[0,k]^d$ is always achieved by a zonotope \cite{DezaManoussakisOnn2018}. In the case of lattice zonotopes, the problem can be reformulated as follows in terms of primitive points.

Let $\mathbb{P}^d$ be the set of the primitive points contained in $\mathbb{Z}^d$. Consider a subset $\mathcal{X}$ of $\mathbb{P}^d$ whose first non-zero coordinate of elements is positive and call
$$
\kappa(\mathcal{X})=\max\left\{\sum_{x\in\mathcal{X}}|x_i|:1\leq{i}\leq{d}\right\}\!\mbox{,}
$$
where the coordinates of a point $x$ from $\mathbb{R}^d$ are denoted by $x_1$ to $x_d$. In other words, $\kappa(\mathcal{X})$ is the sum over all the points in $\mathcal{X}$ of the absolute value of one of their coordinates for which that sum is the greatest. Using this quantity, we can ask the question of \emph{how large the cardinality of $\mathcal{X}$ can be under the requirement that $\kappa(\mathcal{X})$ does not exceed a given integer $k$}. Incidentally, the answer to this primitive point packing problem at the intersection of geometry, number theory, and combinatorics coincides with the largest possible diameter of a lattice zonotope contained in the hypercube $[0,k]^d$. In this article, we settle this problem completely by giving an explicit expression for its solution, thus also solving the question on the largest possible diameter of lattice zonotopes and, conjecturally, the more general one about lattice polytopes.

Let us describe our main results in more details. Following the notation introduced in \cite{DezaPourninSukegawa2020}, we refer to the solution to our packing problem as
$$
\delta_z(d,k)=\max_{\mathcal{X}\in\mathbb{P}^d_\circ}\left\{|\mathcal{X}|:\kappa(\mathcal{X})\leq{k}\right\}\!\mbox{.}
$$

Here $\mathbb{P}^d_\circ$ is the subset of $\mathbb{P}^d$ made up of the points whose first non-zero coordinate is positive. Note that an expression for $\delta_z(d,k)$ is already known when $k$ is less than $2d$. Indeed, it is proven in \cite{DezaManoussakisOnn2018} that, in this case, 
\begin{equation}\label{PPP.sec.0.eq.0}
\delta_z(d,k)=\left\lfloor\frac{(k+1)d}{2}\right\rfloor\!\mbox{.}
\end{equation}

Denote by $B(d,p)$ the ball of radius $p$ for the $1$-norm, centered on the origin of $\mathbb{R}^d$. Note that this ball is a cross-polytope, the polytope dual to the hypercube. Our first result, which we prove in Section \ref{PPP.sec.1} provides the cardinality of the intersection of $B(d,p)$ with $\mathbb{P}^d_\circ$ and the value of $\kappa$ at that intersection.

\begin{thm}\label{PPP.sec.0.thm.1}
For any positive integers $d$ and $p$,
\begin{itemize}
\item[(i)] $\!\displaystyle\left| B(d,p)\cap\mathbb{P}^d_\circ\right|=\frac{1}{2}\sum_{j=1}^{d}2^j{d\choose{j}}\sum_{i=j}^pc_\psi(i,j)$, and
\item[(ii)] $\displaystyle\kappa\!\left(B(d,p)\cap\mathbb{P}^d_\circ\right)=\frac{1}{2}\sum_{j=1}^{d}2^j{d\choose{j}}\sum_{i=j}^pic_\psi(i,j)$,
\end{itemize}
where
$$
\displaystyle c_\psi(p,d)=\frac{1}{(d-1)!}\sum_{i=1}^ds(d,i)J_{i-1}(p)\mbox{.}
$$
\end{thm}

In the statement of this theorem, $s(d,i)$ stands for Stirling's numbers of the first kind and $J_i(p)$ for Jordan's totient function, whose precise definitions will be recalled in Section \ref{PPP.sec.1}. The number of lattice points contained in a lattice polytope is related to certain invariants of the associated toric variety \cite{Oda1988} and expressions for that number have been obtained in particular cases. For instance, a formula for the number of lattice points contained in the dilates by a positive integer of the cross-polytope $B(d,1)$ is known \cite{BeckRobins2015}. In this light, it is noteworthy that Theorem \ref{PPP.sec.0.thm.1} provides an expression for the number of primitive points contained in these dilates. Indeed, the intersection of $B(d,p)$ with $\mathbb{P}^d$ is twice as large as its intersection with $\mathbb{P}^d_\circ$.
%
%

Theorem \ref{PPP.sec.0.thm.1} also solves our primitive point packing problem for special values of $k$. Indeed, it is shown in \cite{DezaPourninSukegawa2020} that the intersection of $B(d,p)$ with $\mathbb{P}^d_\circ$ is the unique subset of $\mathbb{P}^d_\circ$ of cardinality $\delta_z(d,k)$ when
$$
k=\kappa\!\left(B(d,p)\cap\mathbb{P}^d_\circ\right)\!\mbox{.}
$$

Hence, for these values of $k$, Theorem~\ref{PPP.sec.0.thm.1} provides an expression for $\delta_z(d,k)$. Now consider the piecewise-linear map $k\mapsto\lambda(d,k)$ that interpolates $\delta_z(d,k)$ linearly between any two consecutive such values of $k$. It follows from Theorem~\ref{PPP.sec.0.thm.1} that, when $k$ is between $\kappa\!\left(B(d,p-1)\cap\mathbb{P}^d_\circ\right)$ and $\kappa\!\left(B(d,p)\cap\mathbb{P}^d_\circ\right)$ the slope of the interpolation is $d/p$. More precisely, for these values of $k$,
$$
\frac{\lambda(d,k)-\left|B(d,p-1)\cap\mathbb{P}^d_\circ\right|}{k-\kappa\!\left(B(d,p-1)\cap\mathbb{P}^d_\circ\right)}=\frac{d}{p}\mbox{.}
$$

We will prove throughout Sections \ref{PPP.sec.3}, \ref{PPP.sec.3.5}, and \ref{PPP.sec.4} that $\delta_z(d,k)$ is equal to $\lfloor\lambda(d,k){\rfloor}$ except for a few notable, but infrequent cases.

\begin{thm}\label{PPP.sec.0.thm.2}
For any fixed $d$, the maps $k\mapsto\delta_z(d,k)$ and $k\mapsto\lfloor\lambda(d,k){\rfloor}$ coincide, except on an infinite subset $\mathbb{E}$ of $\mathbb{N}\mathord{\setminus}\{0\}$ such that,
$$
\lim_{k\rightarrow\infty}\frac{\left|\mathbb{E}\cap[1,k]\right|}{k^{1/(d-1)}}=0\mbox{.}
$$

In addition, $k\mapsto\delta_z(d,k)$ coincides, on $\mathbb{E}$, with $k\mapsto\lfloor\lambda(d,k){\rfloor}-1$.
\end{thm}

Note that the set $\mathbb{E}$ of exceptions depends on the dimension. We will determine it explicitly in every dimension: it is given a the end of Section \ref{PPP.sec.3.5} when $d>2$, and at the end of Section \ref{PPP.sec.4} when $d=2$. The dependence of $\mathbb{E}$ on $d$ is consistent over all the dimensions above $2$, but as we shall see, its form and its density within $\mathbb{N}$ are slightly different in the $2$-dimensional case.

The proof of theorem \ref{PPP.sec.0.thm.2} will be split into first establishing an upper bound on $\delta_z(d,k)$, and then proving that this bound is exact by constructing sets of primitive points that achieve it. The upper bound is proven in Section \ref{PPP.sec.3}. The general machinery we use to show that this bound is sharp is given in Section~\ref{PPP.sec.3.5} together with sets of primitive points that achieve it in all dimensions greater than $2$. The construction for the $2$-dimensional case is given in Section  \ref{PPP.sec.4}.

Recall that, when $1\leq{k}<2d$, our expression for $\delta_z(d,k)$ simplifies to (\ref{PPP.sec.0.eq.0}). Note, in particular that the bounds of this range for $k$ are $\left|B(d,1)\cap\mathbb{P}^d_\circ\right|$ and $\left|B(d,2)\cap\mathbb{P}^d_\circ\right|$. 
We report in Table~\ref{PPP.sec.3.tab.1} the first values of $\delta_z(d,k)$ such that $k$ is at least $2d$. In this table, the bold numbers are the values of $\delta_z(d,k)$ that coincide, for some integer $p$, with the cardinality of $B(d,p)\cap\mathbb{P}^d_\circ$ and the starred values are the ones such that $k$ belongs to $\mathbb{E}$. Note that, when $d$ is equal to $3$, the smallest number in $\mathbb{E}$ is $135$ and it already lies way outside of the table.

\begin{table}[!t]
\begin{center}
\begin{tabular}{>{\centering}p{0.7cm}cccccccccccccccccc}
\multirow{2}{*}{$d$}&  \multicolumn{15}{c}{$k-2d$}\\
\cline{2-16}
& $0$ & $1$ & $2$ & $3$ & $4$ & $5$ & $6$ & $7$ & $8$ & $9$ & $10$ & $11$ & $12$ & $13$ & $14$\\
\hline
$2$ & $4$ & $5$ & $6$ & $6$ & $7$ & ${\bf 8}$ & $8$ & $8^\star$ & $9$ & $10$ & $10$ & $10^\star$ & $11$ & ${\bf 12}$ & $12$\\
$3$ & $10$ & $11$ & $12$ & $13$ & $14$ & $15$ & $16$ & $17$ & $18$ & $19$ & $20$ & $21$ & $22$ & $23$ & $24$\\
$4$ & $17$ & $18$ & $20$ & $21$ & $22$ & $24$ & $25$ & $26$ & $28$ & $29$ & $30$ & $32$ & $33$ & $34$ & $36$\\
$5$ & $26$ & $28$ & $30$ & $31$ & $33$ & $35$ & $36$ & $38$ & $40$ & $41$ & $43$ & $45$ & $46$ & $48$ & $50$\\
$6$ & $38$ & $40$ & $42$ & $44$ & $46$ & $48$ & $50$ & $52$ & $54$ & $56$ & $58$ & $60$ & $62$ & $64$ & $66$\\
\end{tabular}
\end{center}

\caption{Some values of $\delta_z(d,k)$.}
\label{PPP.sec.3.tab.1}
\end{table}

We will also study the sets of points that solve our primitive point packing problem. 
%
%
%
%
Surprisingly, there exist values of $d$ and $k$ such that $\delta_z(d,k)$ is not achieved by a set of primitive points between $B(d,p-1)\cap\mathbb{P}^d_\circ$ and $B(d,p)\cap\mathbb{P}^d_\circ$ in terms of inclusion, where $p$ is the smallest integer such that $k$ is less than $\kappa\!\left(B(d,p)\cap\mathbb{P}^d_\circ\right)$. We will obtain, in Section \ref{PPP.sec.5}, the following characterization for the values of $d$ and $k$ such that this phenomenon occurs. Note in particular, that this phenomenon only happens when $d$ is equal to $2$.

\begin{thm}\label{PPP.sec.0.thm.4}
Consider the smallest integer $p$ such that $k<\kappa\!\left(B(d,p)\cap\mathbb{P}^d_\circ\right)$. If $d$ is greater than $2$, then any subset $\mathcal{X}$ of $\mathbb{P}^d_\circ$ of cardinality $\delta_z(d,k)$ such that $\kappa(\mathcal{X})$ does not exceed $k$ must satisfy
\begin{equation}\label{PPP.sec.0.thm.4.eq.0}
B(d,p-1)\cap\mathbb{P}^d_\circ\subset\mathcal{X}\subset{B(d,p)\cap\mathbb{P}^d_\circ}\mbox{.}
\end{equation}

If, on the other hand, $d$ is equal to $2$, then there exists a subset $\mathcal{X}$ of $\mathbb{P}^d_\circ$ such that $|\mathcal{X}|=\delta_z(d,k)$ and $\kappa(\mathcal{X})\leq{k}$ that satisfies (\ref{PPP.sec.0.thm.4.eq.0}) if and only if $p=2$, or $p$ is odd, or $p/2$ is even, or the following two conditions hold.
\begin{itemize}
\item[(i)] $k-\kappa\!\left(B(2,p-1)\cap\mathbb{P}^2_\circ\right)$ is distinct from $p/2+1$,
\item[(ii)] $\kappa\!\left(B(2,p)\cap\mathbb{P}^2_\circ\right)-k$ is distinct from $p/2-1$.
\end{itemize}
\end{thm}

We further show in Section \ref{PPP.sec.5} that the unicity result from \cite{DezaPourninSukegawa2020} that we mention above singles out every value of $k$ such that there is a unique set of primitive points that solves our packing problem.
\begin{thm}\label{PPP.sec.0.thm.5}
There exists a unique subset $\mathcal{X}$ of $\mathbb{P}^d_\circ$ such that $|\mathcal{X}|=\delta_z(d,k)$ and $\kappa(\mathcal{X})\leq{k}$ if and only if $k=\kappa\!\left(B(d,p)\cap\mathbb{P}^d_\circ\right)$ for some $p$.
\end{thm}

We conclude the article, by gathering in Section \ref{PPP.sec.1.5} a number of asymptotic estimates. Among them, we derive from Theorem \ref{PPP.sec.0.thm.1} asymptotic estimates for $\left|B(d,p)\cap\mathbb{P}^d_\circ\right|$ in terms of $\kappa(B(d,p)\cap\mathbb{P}^d_\circ)$ when $p$ is fixed and $d$ goes to infinity. This complements the result from \cite{DezaPourninSukegawa2020} on the asymptotic behavior of $\delta_z(d,k)$. We also obtain an exact asymptotic estimate of the number of primitive points of $1$-norm $p$ contained in the positive orthant $]0,+\infty[^d$, for any fixed $d$, when $p$ goes to infinity. Note that these points are a special case of integer compositions \cite{MacMahon1893}, the compositions of $p$ into $d$ relatively prime integers.
 
\section{The number of primitive points in a cross-polytope}\label{PPP.sec.1}

Consider the isometry $\sigma$ of $\mathbb{R}^d$ that permutes the coordinates cyclically as
$$
\sigma:
\left[
\begin{array}{c}
x_1\\
x_2\\
\vdots\\
x_{d-1}\\
x_d
\end{array}
\right]
\mapsto
\left[
\begin{array}{c}
x_d\\
x_1\\
\vdots\\
x_{d-2}\\
x_{d-1}\\
\end{array}
\right]\!\mbox{.}
$$

Further consider the map $\tau$ that sends a vector $x$ of $\mathbb{R}^d$ to $\sigma(x)$ when $x_d$ is non-negative and to $-\sigma(x)$ otherwise. In the remainder of the article, $C$ denotes the cyclic group generated by $\tau$. Observe that $C$ has order $d$. For any point $x$ contained in $\mathbb{R}^d$, we denote by $C\mathord{\cdot}x$ the orbit of $x$ under the action of $C$ on $\mathbb{R}^d$. Note that all the elements of $C\mathord{\cdot}x$ have the same $1$-norm and, when $x$ is primitive, $C\mathord{\cdot}x$ is a subset of $\mathbb{P}^d$. In addition, observe that the intersection of $B(d,p)$ and $\mathbb{P}^d_\circ$ is invariant under the action of $C$ on $\mathbb{R}^d$.

We first recall a proposition on the relation between $\kappa(\mathcal{X})$ and the $1$-norms of the elements of $\mathcal{X}$, that is used implicitly in \cite{DezaPourninSukegawa2020}.

\begin{prop}\label{PPP.sec.3.prop.1}
For any subset $\mathcal{X}$ of $\mathbb{P}^d_\circ$,
$$
\kappa(\mathcal{X})\geq\frac{1}{d}\sum_{x\in\mathcal{X}}\|x\|_1\mbox{,}
$$
with equality if and only if $\mathcal{X}$ is invariant under the action of $C$ on $\mathbb{R}^d$.
\end{prop}

Now denote by $c_\psi(p,d)$ the number of primitive points of $1$-norm $p$ and dimension $d$ contained in the positive orthant:
$$
c_\psi(p,d)=\left|\left\{x\in\mathbb{P}^d\cap\left]0,+\infty\right[^d:\|x\|_1=p\right\}\right|\!\mbox{.}
$$

This quantity makes it possible to express as follows the sum over the intersection $B(d,p)\cap\mathbb{P}^d_\circ$ of any map that is well-behaved relative to the $1$-norm.
%

\begin{lem}\label{PPP.sec.1.lem.1}
For any map $f:\mathbb{N}\rightarrow\mathbb{R}$ and any positive integer $p$,
$$
\sum{f\!\left(\|x\|_1\right)}=\sum_{j=1}^{d}2^j{d\choose{j}}\sum_{i=j}^pf(i)c_\psi(i,j)\mbox{,}
$$
where the sum in the left-hand side is over the primitive points $x$ in $B(d,p)$.
\end{lem}
\begin{proof}
Consider the sphere $S(d,i)$ of radius $i$ for the $1$-norm, centered at the origin of $\mathbb{R}^d$. In other words, $S(d,i)$ is the boundary of $B(d,i)$.

Let us first count the number of primitive points contained in $S(d,i)$. As $B(d,i)$ is a cross polytope, $c_\psi(i,j)$ is precisely the number of primitive points contained in the relative interior of one of its faces of dimension $j-1$.

We therefore immediately obtain
$$
\left| S(d,i)\cap\mathbb{P}^d\right|=\sum_{j=1}^d2^j{d\choose{j}}c_\psi(i,j)\mbox{.}
$$
where the coefficient of $c_\psi(i,j)$ in the right-hand side is the number of the faces of dimension $j-1$ of a cross-polytope.

Now observe that by construction, the intersection of $B(d,p)$ with $\mathbb{P}^d$ is partitioned into the sets $S(d,i)\cap\mathbb{P}^d$ where $i$ ranges from $1$ to $p$. Hence,
$$
\sum{f\!\left(\|x\|_1\right)}=\sum_{i=1}^{p}f(i)\sum_{j=1}^d2^j{d\choose{j}}c_\psi(i,j)\mbox{,}
$$
where the sum in the left-hand side is over the primitive points $x$ in $B(d,p)$. Exchanging the two sums in the right-hand side and noticing that $c_\psi(i,j)$ vanishes when $i$ is less than $j$ completes the proof.
\end{proof}

Note that the cardinality of the intersection $B(d,p)\cap\mathbb{P}^d$ can be obtained from Lemma \ref{PPP.sec.1.lem.1} by using, for $f$, the map that sends every natural integer to $1$. By symmetry, the cardinality of $B(d,p)\cap\mathbb{P}^d_\circ$ is half of that quantity. In addition, since the intersection of $B(d,p)$ with $\mathbb{P}^d_\circ$ is invariant under the action of $C$ on $\mathbb{R}^d$, it follows from Proposition \ref{PPP.sec.3.prop.1} that
$$
d\kappa\!\left(B(d,p)\cap\mathbb{P}^d_\circ\right)=\sum\|x\|_1\mbox{,}
$$
where the sum in the right-hand side is over the points $x$ in that intersection. As above, that sum is half the sum of the $1$-norms of the points in $B(d,p)\cap\mathbb{P}^d$, that can be obtained from Lemma \ref{PPP.sec.1.lem.1} by using for $f$, the identity map on $\mathbb{N}$. According to this discussion, we obtain the following.
\begin{prop}\label{PPP.sec.1.prop.1}
For any positive integers $d$ and $p$,
\begin{itemize}
\item[(i)] $\!\displaystyle\left| B(d,p)\cap\mathbb{P}^d_\circ\right|=\frac{1}{2}\sum_{j=1}^{d}2^j{d\choose{j}}\sum_{i=j}^pc_\psi(i,j)$, and
\item[(ii)] $\displaystyle\kappa\!\left(B(d,p)\cap\mathbb{P}^d_\circ\right)=\frac{1}{2d}\sum_{j=1}^{d}2^j{d\choose{j}}\sum_{i=j}^pic_\psi(i,j)$.
\end{itemize}
\end{prop}

It is noteworthy that, as a consequence of this proposition,
$$
\kappa\!\left(B(d,p)\cap\mathbb{P}^d_\circ\right)-\kappa\!\left(B(d,p-1)\cap\mathbb{P}^d_\circ\right)=\frac{p\!\left|S(d,p)\cap\mathbb{P}^d_\circ\right|}{d}\mbox{,}
$$
where $S(d,p)$ is the sphere of radius $p$ for the $1$-norm, centered at the origin of $\mathbb{R}^d$. 
%
%
We now derive a formula for $c_\psi(p,d)$. Let us denote by $J_q$ Jordan's totient function; that is, for any two positive integers $p$ and $q$,
$$
J_q(p)=p^q\prod_{n|p}\left(1-\frac{1}{n^q}\right)\!\mbox{,}
$$
where, in this expression, $n$ ranges over prime numbers. We also denote by $s(d,i)$ Stirling's numbers of the first kind. Let us recall that these numbers can be computed using the recurrence
$$
s(d+1,i)=-ds(d,i)+s(d,i-1)\mbox{,}
$$
with $s(i,i)=1$ for all $i\in\mathbb{N}$ and $s(d,0)=0$ when $d$ is positive.

\begin{thm}\label{PPP.sec.1.thm.1}
$\displaystyle c_\psi(p,d)=\frac{1}{(d-1)!}\sum_{i=1}^ds(d,i)J_{i-1}(p)$.
\end{thm}
\begin{proof}
First observe that
$$
{p-1\choose{d-1}}=\sum_{q|p}c_\psi\!\left(\frac{p}{q},d\right)\!\mbox{.}
$$

Indeed, the left-hand side of this equality is the number of the lattice points of $1$-norm $p$ contained in the positive orthant $]0,+\infty[^d$ or, equivalently, the number of compositions of $p$ into $d$ integers. Dividing each of these points by the greatest common divisor of its coordinates results in a set of primitive points whose number is the right-hand side of the equality.

By M{\"o}bius' inversion formula,
\begin{equation}\label{PPP.sec.1.thm.1.eq.1}
c_\psi(p,d)=\sum_{q|p}\mu(q){p/q-1\choose{d-1}}
\end{equation}
where $\mu$ denotes M{\"o}bius' function. By definition of Stirling's numbers of the first kind, the following holds when $d$ is positive.
\begin{equation}\label{PPP.sec.1.thm.1.eq.2}
{p/q-1\choose{d-1}}=\frac{1}{(d-1)!}\sum_{i=1}^ds(d,i)\!\left(\frac{p}{q}\right)^{i-1}\mbox{.}
\end{equation}

Combining (\ref{PPP.sec.1.thm.1.eq.1}) with (\ref{PPP.sec.1.thm.1.eq.2}), and permuting the sums yields
$$
c_\psi(p,d)=\frac{1}{(d-1)!}\sum_{i=1}^ds(d,i)\sum_{q|p}\mu(q)\!\left(\frac{p}{q}\right)^{i-1}\mbox{.}
$$

Since Jordan's totient function satisfies
$$
\sum_{q|p}J_i(q)=p^i\mbox{,}
$$
M{\"o}bius' inversion formula provides the desired result.
\end{proof}

Theorem \ref{PPP.sec.0.thm.1} immediately follows from Proposition \ref{PPP.sec.1.prop.1} and Theorem \ref{PPP.sec.1.thm.1}.

\section{Upper bounds on $\delta_z(d,k)$}\label{PPP.sec.3}

%
%

We first establish the following upper bound, and provide some information on the subsets of $\mathbb{P}^d_\circ$ that achieve it in a special case. 

\begin{lem}\label{PPP.sec.3.lem.0.5}
Consider a subset $\mathcal{X}$ of $\mathbb{P}^d_\circ$. If $\kappa(\mathcal{X})$ is at most $k$, then $\mathcal{X}$ has cardinality at most $\lfloor\lambda(d,k)\rfloor$. If in addition,
\begin{itemize}
\item[(i)] $\mathcal{X}$ has cardinality exactly $\lfloor\lambda(d,k)\rfloor$,
\item[(ii)] $d$ is a proper divisor of $p$,
\item[(iii)] $p/d$ a divisor of $k-\kappa(B(d,p-1)\cap\mathbb{P}^d_\circ)$,
\end{itemize}
where $p$ is the smallest integer such that $k<\kappa(B(d,p)\cap\mathbb{P}^d_\circ)$, then
$$
B(d,p-1)\cap\mathbb{P}^d_\circ\subset\mathcal{X}\subset{B(d,p)\cap\mathbb{P}^d_\circ}\mbox{.}
$$
\end{lem}
\begin{proof}
Assume that $\kappa(\mathcal{X})\leq{k}$. We further assume without loss of generality that $|\mathcal{X}|$ is equal to $\delta_z(d,k)$.  Let $p$ be the smallest integer such that $k$ is less than $\kappa(B(d,p)\cap\mathbb{P}^d_\circ)$. It follows that $\kappa(B(d,p-1)\cap\mathbb{P}^d_\circ)\leq{k}$ and therefore,
$$
\left|B(d,p-1)\cap\mathbb{P}^d_\circ\right|\leq\delta_z(d,k)\mbox{.}
$$

Since the right-hand side of this inequality is also the cardinality of $\mathcal{X}$, we can consider a subset $\mathcal{X}^\star$ of $\mathcal{X}$ with the same cardinality than $B(d,p-1)\cap\mathbb{P}^d_\circ$. We can also require that the sum of the $1$-norms of the points in $\mathcal{X}^\star$ is as small as possible. We then consider a bijection $\psi:B(d,p-1)\cap\mathbb{P}^d_\circ\rightarrow\mathcal{X}^\star$. We can assume without loss of generality that the restriction of $\psi$ to $B(d,p-1)\cap\mathcal{X}$ is the identity. By this assumption, the $1$-norm of an element $y$ of $B(d,p-1)\cap\mathbb{P}^d_\circ$ is at most that of its image by $\psi$. Indeed, either $\psi(y)$ coincides with $y$ or it is outside of $B(d,p-1)$. In the latter case, the $1$-norm of $\psi(y)$ is at least $p$, while the $1$-norm of $y$ is at most $p-1$. According to Proposition \ref{PPP.sec.3.prop.1},
$$
\kappa(\mathcal{X})\geq\frac{1}{d}\sum_{x\in\mathcal{X}}\|x\|_1
$$
and, as $B(d,p-1)\cap\mathbb{P}^d_\circ$ is invariant under the action of $C$ on $\mathbb{R}^d$,
$$
\kappa\!\left(B(d,p-1)\cap\mathbb{P}^d_\circ\right)=\frac{1}{d}\sum\|y\|_1\mbox{,}
$$
where the sum in the right-hand side of this equality is over the points $y$ contained in $B(d,p-1)\cap\mathbb{P}^d_\circ$. Since $\psi$ cannot make $1$-norms decrease, 
$$
\kappa(\mathcal{X})\geq\kappa\!\left(B(d,p-1)\cap\mathbb{P}^d_\circ\right)+\frac{1}{d}\sum_{\substack{x\in\mathcal{X}\\x\not\in\mathcal{X}^\star}}\|x\|_1\mbox{.}
$$

In addition, if $B(d,p-1)\cap\mathbb{P}^d_\circ$ and $\mathcal{X}^\star$ do not coincide, then this inequality is strict.
As all the points in $\mathcal{X}\mathord{\setminus}\mathcal{X}^\star$ have $1$-norm at least $p$,
\begin{equation}\label{PPP.sec.3.lem.0.5.eq.1}
\kappa(\mathcal{X})\geq\kappa\!\left(B(d,p-1)\cap\mathbb{P}^d_\circ\right)+p\frac{|\mathcal{X}|-\left|B(d,p-1)\cap\mathbb{P}^d_\circ\right|}{d}\mbox{.}
\end{equation}

Note that this inequality is strict as soon as a point in $\mathcal{X}\mathord{\setminus}\mathcal{X}^\star$ has $1$-norm greater than $p$ or $B(d,p-1)\cap\mathbb{P}^d_\circ$ does not coincide with $\mathcal{X}^\star$. As $\kappa(\mathcal{X})\leq{k}$ and $|\mathcal{X}|=\delta_z(d,k)$, we obtain from (\ref{PPP.sec.3.lem.0.5.eq.1}) the desired upper bound on $|\mathcal{X}|$.

Now assume that $d$ is a proper divisor of $p$ and that $p/d$ is a divisor of $k-\kappa(B(d,p-1)\cap\mathbb{P}^d_\circ)$. Observe that, in this case, $\lambda(d,k)$ is an integer. Therefore, if $\mathcal{X}$ has cardinality $\lfloor\lambda(d,k)\rfloor$, then (\ref{PPP.sec.3.lem.0.5.eq.1}) must be an equality. In this case, as we discussed above, $B(d,p-1)\cap\mathbb{P}^d_\circ$ must coincide with $\mathcal{X}^\star$. In other words, $B(d,p-1)\cap\mathbb{P}^d_\circ$ is a subset of $\mathcal{X}$. In addition, all the points in $\mathcal{X}\mathord{\setminus}\mathcal{X}^\star$ must have $1$-norm at most $p$. It follows that $\mathcal{X}$ is a subset of $B(d,p)\cap\mathbb{P}^d_\circ$.
\end{proof}

It is an immediate consequence, of Lemma \ref{PPP.sec.3.lem.0.5} that $\delta_z(d,k)\leq\lfloor\lambda(d,k)\rfloor$ for all $d$ and $k$. We can refine this bound as follows.
%

\begin{lem}\label{PPP.sec.3.lem.1}
Consider the smallest integer $p$ such that $k<\kappa\!\left(B(d,p)\cap\mathbb{P}^d_\circ\right)$. If $d$ is a proper divisor of $p$ and either
\begin{itemize}
\item[(i)] $k-\kappa\!\left(B(d,p-1)\cap\mathbb{P}^d_\circ\right)$ is equal to $p/d$, or
\item[(ii)] $\kappa\!\left(B(d,p)\cap\mathbb{P}^d_\circ\right)-k$ is equal to $p/d$,
\end{itemize}
then $\delta_z(d,k)$ cannot be equal to $\lfloor\lambda(d,k)\rfloor$.
\end{lem}
\begin{proof}
Consider a subset $\mathcal{X}$ of $\mathbb{P}^d_\circ$ such that $\kappa(\mathcal{X})$ is at most $k$. Assume that $\mathcal{X}$ has cardinality exactly $\lfloor\lambda(d,k)\rfloor$. Hence, by the definition of $\lambda(d,k)$,
\begin{equation}\label{PPP.sec.3.lem.1.eq.1}
|\mathcal{X}|=\left|B(d,p-1)\cap\mathbb{P}^d_\circ\right|+\left\lfloor\frac{k-\kappa\!\left(B(d,p-1)\cap\mathbb{P}^d_\circ\right)}{p/d}\right\rfloor\!\mbox{.}
\end{equation}

Further assume that $d$ is a proper divisor of $p$ and let us reach a contradiction in each of the two special cases considered in the statement of the lemma. 

We first treat the case when
\begin{equation}\label{PPP.sec.3.lem.1.eq.2}
k-\kappa\!\left(B(d,p-1)\cap\mathbb{P}^d_\circ\right)=\frac{p}{d}\mbox{.}
\end{equation}

In this case, by Lemma \ref{PPP.sec.3.lem.0.5}, $B(d,p-1)\cap\mathbb{P}^d_\circ$ is a subset of $\mathcal{X}$, which in turn is contained in $B(d,p)$. Combining (\ref{PPP.sec.3.lem.1.eq.1}) with (\ref{PPP.sec.3.lem.1.eq.2}), we obtain that
$$
|\mathcal{X}|=\left|B(d,p-1)\cap\mathbb{P}^d_\circ\right|+1\mbox{.}
$$

In other words, $\mathcal{X}$ contains exactly one element $x$ of $B(d,p)\cap\mathbb{P}^d_\circ$. According to (\ref{PPP.sec.3.lem.1.eq.2}) and to Proposition \ref{PPP.sec.3.prop.1}, the absolute value of all the coordinates of $x$ must be equal to $p/d$. Since $d$ is a proper divisor of $p$, $x$ cannot be primitive. This contradiction shows that $\delta_z(d,k)$ cannot coincide with $\lfloor\lambda(d,k)\rfloor$.

Now, assume that $d$ is a proper divisor of $p$ and that
\begin{equation}\label{PPP.sec.3.lem.1.eq.3}
\kappa\!\left(B(d,p)\cap\mathbb{P}^d_\circ\right)-k=\frac{p}{d}\mbox{.}
\end{equation}

Recall that, according to Proposition \ref{PPP.sec.1.prop.1},
\begin{equation}\label{PPP.sec.3.lem.1.eq.4}
\kappa\!\left(B(d,p)\cap\mathbb{P}^d_\circ\right)-\kappa\!\left(B(d,p-1)\cap\mathbb{P}^d_\circ\right)=\frac{p\!\left|S(d,p)\cap\mathbb{P}^d_\circ\right|}{d}\mbox{.}
\end{equation}

Combining the equalities (\ref{PPP.sec.3.lem.1.eq.1}), (\ref{PPP.sec.3.lem.1.eq.3}), and (\ref{PPP.sec.3.lem.1.eq.4}) yields
$$
|\mathcal{X}|=\left|B(d,p)\cap\mathbb{P}^d_\circ\right|-1\mbox{.}
$$

Again, it follows from Lemma \ref{PPP.sec.3.lem.0.5} that $B(d,p-1)\cap\mathbb{P}^d_\circ$ is a subset of $\mathcal{X}$, and $\mathcal{X}$ a subset of $B(d,p)\cap\mathbb{P}^d_\circ$. Hence, $\mathcal{X}$ is obtained by removing a single point of $1$-norm $p$ from $B(d,p)\cap\mathbb{P}^d_\circ$. According to (\ref{PPP.sec.3.lem.1.eq.3}) and to Proposition~\ref{PPP.sec.3.prop.1}, the absolute value of all the coordinates of $x$ must be equal to $p/d$ and, since $d$ is a proper divisor of $p$, $x$ cannot be a primitive point. It follows from this contradiction, that $\delta_z(d,k)$ is not equal to $\lfloor\lambda(d,k)\rfloor$.
\end{proof}

In the $2$-dimensional case, we can further refine our bound as follows.

\begin{lem}\label{PPP.sec.3.lem.2}
Consider the  smallest integer $p$ such that $k<\kappa(B(2,p)\cap\mathbb{P}^2_\circ)$. If $p$ is a multiple of $4$ and $k-\kappa(B(2,p-1)\cap\mathbb{P}^2_\circ)$ is an odd multiple of $p/2$, then $\delta_z(2,k)$ cannot be equal to $\lfloor\lambda(d,k)\rfloor$.
\end{lem}
\begin{proof}
Assume that $p$ is a multiple of $4$ and $k-\kappa(B(2,p-1)\cap\mathbb{P}^2_\circ)$ an odd multiple of $p/2$. Consider a subset $\mathcal{X}$ of $\mathbb{P}^d_\circ$ such that $\kappa(\mathcal{X})$ is not greater than $k$. We proceed as in the proof of Lemma \ref{PPP.sec.3.lem.1} by assuming that
\begin{equation}\label{PPP.sec.3.lem.2.eq.1}
|\mathcal{X}|=\left|B(2,p-1)\cap\mathbb{P}^2_\circ\right|+\left\lfloor\frac{k-\kappa\!\left(B(2,p-1)\cap\mathbb{P}^2_\circ\right)}{p/2}\right\rfloor\!\mbox{,}
\end{equation}
and we aim for a contradiction.

According to Lemma \ref{PPP.sec.3.lem.0.5}, $\mathcal{X}$ admits $B(2,p-1)\cap\mathbb{P}^2_\circ$ as a subset and is itself a subset of $B(2,p)\cap\mathbb{P}^2_\circ$. Denote by $\mathcal{Y}$ the set of the points of $1$-norm $p$ contained in $\mathcal{X}$. Since $k-\kappa(B(2,p-1)\cap\mathbb{P}^2_\circ)$ an odd multiple of $p/2$, it follows from (\ref{PPP.sec.3.lem.2.eq.1}) that the cardinality of $\mathcal{Y}$ is odd. It also follows from (\ref{PPP.sec.3.lem.2.eq.1}) that
$$
|\mathcal{Y}|=\frac{k-\kappa\!\left(B(2,p-1)\cap\mathbb{P}^2_\circ\right)}{p/2}\mbox{.}
$$

Therefore, as $\kappa(\mathcal{X})\leq{k}$ and as the points in $\mathcal{Y}$ have $1$-norm $p$,
$$
\kappa(\mathcal{X})\leq\kappa\!\left(B(2,p-1)\cap\mathbb{P}^2_\circ\right)+\frac{1}{2}\sum_{x\in\mathcal{Y}}\|x\|_1\mbox{.}
$$

Observe that, by Proposition~\ref{PPP.sec.3.prop.1} the opposite inequality holds. Hence, according to the same proposition, $\mathcal{X}$ is invariant under the action of $C$ on $\mathbb{R}^2$. Since $\mathcal{X}\mathord{\setminus}\mathcal{Y}$ is also invariant under that action, then so is $\mathcal{Y}$. Therefore,
$$
\sum_{x\in\mathcal{Y}}|x_1|=\frac{p|\mathcal{Y}|}{2}\mbox{.}
$$

However, since $p$ is a multiple of $4$, the right-hand side of this equality is even. As $|\mathcal{Y}|$ is odd, there must exist a point $x$ in $\mathcal{Y}$ such that $x_1$ is even. However, since the $1$-norm of $x$ is even, the coordinates of $x$ must both be even. It follows that $x$ cannot be primitive, a contradiction.
\end{proof}

\section{An expression for $\delta_z(d,k)$ when $d\geq3$}\label{PPP.sec.3.5}

In this section, we build a subset of $\mathbb{P}^d_\circ$ that achieves the bound on $\delta_z(d,k)$ provided by Lemmas~\ref{PPP.sec.3.lem.0.5} and \ref{PPP.sec.3.lem.1} when $d>2$ and for a number of values of $k$ when $d=2$. We will distinguish two cases depending on whether $d$ is, or not, a divisor of the smallest integer $p$ such that $k<\kappa(B(d,p)\cap\mathbb{P}^d_\circ)$. Throughout the section, we assume that $k$ and $d$ (and therefore $p$) are fixed.

Let us first address the case when $d$ is not a proper divisor of $p$. Note that, in this case, our construction will be done for all $d\geq2$. We introduce, for the purpose of the construction, $d$ sets that we denote $\mathcal{S}_1$ to $\mathcal{S}_d$, made up of points of $1$-norm $p$ from $\mathbb{P}^d_\circ$. Denote by $r$ the remainder of the integer division of $p$ by $d$ and recall that $\lfloor{p/d}\rfloor$ is the quotient of this division. As a consequence, the lattice point $x$ such that $x_i=\lceil{p/d}\rceil$ when $i\leq{r}$ and $x_i=\lfloor{p/d}\rfloor$ otherwise has $1$-norm $p$. Moreover, $\lceil{p/d}\rceil$ and $\lfloor{p/d}\rfloor$ are relatively prime as they differ by $1$, and $x$ is necessarily primitive. For any integer $j$ such that $1\leq{j}\leq{d}$, denote
$$
\mathcal{S}_j=\{\tau^{ir}(x):1\leq{i}\leq{j}\}\mbox{.}
$$

As announced, $\mathcal{S}_j$ is made up of $j$ points of $1$-norm $p$ from $\mathbb{P}^d_\circ$. Moreover,
\begin{equation}\label{PPP.sec.3.5.eq.1}
\kappa\!\left(\mathcal{S}_j\right)=\left\lceil\frac{jp}{d}\right\rceil\!\mbox{.}
\end{equation}

Using the sets $\mathcal{S}_1$ to $\mathcal{S}_d$, we prove the following.

\begin{lem}\label{PPP.sec.3.5.lem.1}
If $d$ is not a divisor of $p$, then there exists a subset $\mathcal{X}$ of $\mathbb{P}^d_\circ$ such that $\kappa(\mathcal{X})$ is at most $k$ and $|\mathcal{X}|$ is equal to $\lfloor\lambda(d,k)\rfloor$.
\end{lem}
\begin{proof}
Assume that $d$ is not a divisor of $p$ and consider a point $x$ of $1$-norm $p$ in $\mathbb{P}^d_\circ$. Since $C$ has order $d$, $C\mathord{\cdot}x$ has cardinality at most $d$. In particular, the $1$-norms of the elements of $C\mathord{\cdot}x$ sum to at most $dp$. Observe that $C\mathord{\cdot}x$ coincides with $\mathcal{S}_d$ when $x$ is the unique element of $\mathcal{S}_1$, and that $\mathcal{S}_d$ has cardinality exactly $d$. Further observe that the points of $1$-norm $p$ in $\mathbb{P}^d_\circ$ are partitioned into the orbits of their elements under the action of $C$ on $\mathbb{R}^d$. Hence, one can take the union of some of these orbits in order to build a set $\mathcal{Y}$ of points of $1$-norm $p$ from $\mathbb{P}^d_\circ$ disjoint from $\mathcal{S}_d$, invariant under the action of $C$ on $\mathbb{R}^d$, and that satisfies
$$
\frac{k-\kappa\!\left(B(d,p-1)\cap\mathbb{P}^d_\circ\right)}{p/d}-d<\frac{1}{p}\sum_{x\in\mathcal{Y}}\|x\|_1\leq\frac{k-\kappa\!\left(B(d,p-1)\cap\mathbb{P}^d_\circ\right)}{p/d}\mbox{.}
$$

In this case, the integer $j$ defined as
\begin{equation}\label{PPP.sec.3.5.lem.1.eq.1}
j=\left\lfloor\frac{k-\kappa\!\left(B(d,p-1)\cap\mathbb{P}^d_\circ\right)}{p/d}\right\rfloor-\frac{1}{p}\sum_{x\in\mathcal{Y}}\|x\|_1
\end{equation}
is non-negative and less than $d$. Denote
$$
\mathcal{X}=\left(B(d,p-1)\cap\mathbb{P}^d_\circ\right)\cup\mathcal{Y}\cup\mathcal{S}_j\mbox{,}
$$
with the convention that $\mathcal{S}_0$ is the empty set. By construction, both $\mathcal{Y}$ and the intersection of $B(d,p-1)$ with $\mathbb{P}^d_\circ$ are invariant under the action of $C$ on $\mathbb{R}^d$. In particular, according to Proposition \ref{PPP.sec.3.prop.1},
\begin{equation}\label{PPP.sec.3.5.lem.1.eq.2}
\kappa(\mathcal{X})=\kappa\!\left(B(d,p-1)\cap\mathbb{P}^d_\circ\right)+\frac{1}{d}\sum_{x\in\mathcal{Y}}\|x\|_1+\kappa(S_j)\mbox{.}
\end{equation}

However, it follows from (\ref{PPP.sec.3.5.lem.1.eq.1}) that
\begin{equation}\label{PPP.sec.3.5.lem.1.eq.3}
\frac{1}{d}\sum_{x\in\mathcal{Y}}\|x\|_1\leq{k-\kappa\!\left(B(d,p-1)\cap\mathbb{P}^d_\circ\right)-\frac{jp}{d}}\mbox{.}
\end{equation}

In addition, by (\ref{PPP.sec.3.5.eq.1}),
\begin{equation}\label{PPP.sec.3.5.lem.1.eq.4}
\kappa(\mathcal{S}_j)<\frac{jp}{d}+1\mbox{,}
\end{equation}

As this inequality is strict, we obtain $\kappa(\mathcal{X})\leq{k}$ by combining (\ref{PPP.sec.3.5.lem.1.eq.2}), (\ref{PPP.sec.3.5.lem.1.eq.3}), and (\ref{PPP.sec.3.5.lem.1.eq.4}). Now recall that all the points in $\mathcal{Y}$ have $1$-norm $p$ and $\mathcal{S}_j$ has cardinality $j$. As a consequence the cardinality of $\mathcal{X}$ can be decomposed as
\begin{equation}\label{PPP.sec.3.5.lem.1.eq.5}
|\mathcal{X}|=\left|B(d,p-1)\cap\mathbb{P}^d_\circ\right|+\frac{1}{p}\sum_{x\in\mathcal{Y}}\|x\|_1+j\mbox{.}
\end{equation}

Combining (\ref{PPP.sec.3.5.lem.1.eq.1}) and (\ref{PPP.sec.3.5.lem.1.eq.5}) yields
$$
|\mathcal{X}|=\left|B(d,p-1)\cap\mathbb{P}^d_\circ\right|+\left\lfloor\frac{k-\kappa\!\left(B(d,p-1)\cap\mathbb{P}^d_\circ\right)}{p/d}\right\rfloor\!\mbox{.}
$$

The right-hand side of this equality is $\lfloor\lambda(d,k)\rfloor$, as desired.
\end{proof}

We now address the case when $d$ is a divisor of $p$. In this case, the construction of subset of $\mathbb{P}^d_\circ$ of the right cardinality will rely on the following lemma.

\begin{lem}\label{PPP.sec.3.5.lem.2}
Assume that $d$ is a divisor of $p$ and that
\begin{equation}\label{PPP.sec.3.5.lem.2.eq.0}
\kappa\!\left(B(d,p-1)\cap\mathbb{P}^d_\circ\right)+\frac{2p}{d}\leq{k}<\kappa\!\left(B(d,p)\cap\mathbb{P}^d_\circ\right)-\frac{p}{d}\mbox{.}
\end{equation}

Let $\mathcal{Y}$ be the union of the orbits under the action of $C$ on $\mathbb{R}^d$ of some points of $1$-norm $p$ from $\mathbb{P}^d_\circ$. If $|\mathcal{Y}|$ is at least $d+3$ and, whenever $2\leq{j}\leq|\mathcal{Y}|-2$, there exists a subset $\mathcal{S}$ of $j$ elements of $\mathcal{Y}$ satisfying
\begin{equation}\label{PPP.sec.3.5.lem.2.eq.0.5}
\kappa(\mathcal{S})=\frac{jp}{d}\mbox{,}
\end{equation}
then there exists a subset $\mathcal{X}$ of $\mathbb{P}^d_\circ$ such that $\kappa(\mathcal{X})\leq{k}$ and $|\mathcal{X}|=\lfloor\lambda(d,k)\rfloor$.
\end{lem}

\begin{proof}
Assume that $\mathcal{Y}$ has cardinality at least $d+3$. 
%
%
%
%
Consider the set of the points of $1$-norm $p$ contained in $\mathbb{P}^d_\circ$. Observe that this set is partitioned into the orbits of its elements under the action of $C$ on $\mathbb{R}^d$. Moreover, these orbits have cardinality at most $d$. As in addition, $|\mathcal{Y}|\geq{d+3}$, one can take the union of some of these orbits in order to build a set $\mathcal{Z}$ disjoint from $\mathcal{Y}$ that is invariant under the action of $C$ on $\mathbb{R}^d$ and that satisfies the inequalities
$$
\frac{k-\kappa\!\left(B(d,p-1)\cap\mathbb{P}^d_\circ\right)}{p/d}-|\mathcal{Y}|+1<\frac{1}{p}\sum_{x\in\mathcal{Z}}\|x\|_1\leq\frac{k-\kappa\!\left(B(d,p-1)\cap\mathbb{P}^d_\circ\right)}{p/d}-2\mbox{.}
$$

As an immediate consequence, the integer $j$ defined as
\begin{equation}\label{PPP.sec.3.5.lem.2.eq.1}
j=\left\lfloor\frac{k-\kappa\!\left(B(d,p-1)\cap\mathbb{P}^d_\circ\right)}{p/d}\right\rfloor-\frac{1}{p}\sum_{x\in\mathcal{Z}}\|x\|_1
\end{equation}
is such that $2\leq{j}\leq{|\mathcal{Y}|-2}$. Assume that there exists a subset $\mathcal{S}$ of $j$ elements of $\mathcal{Y}$ that satisfies (\ref{PPP.sec.3.5.lem.2.eq.0.5}) and consider the set
$$
\mathcal{X}=\left[B(d,p-1)\cap\mathbb{P}^d_\circ\right]\cup\mathcal{Z}\cup\mathcal{S}\mbox{.}
$$

From there on, the argument is the same as in the proof of Lemma \ref{PPP.sec.3.5.lem.1}. According to Proposition \ref{PPP.sec.3.prop.1},
\begin{equation}\label{PPP.sec.3.5.lem.2.eq.2}
\kappa(\mathcal{X})=\kappa\!\left(B(d,p-1)\cap\mathbb{P}^d_\circ\right)+\frac{1}{d}\sum_{x\in\mathcal{Z}}\|x\|_1+\kappa(\mathcal{S})\mbox{.}
\end{equation}

Moreover, it follows from (\ref{PPP.sec.3.5.lem.2.eq.1}) that
\begin{equation}\label{PPP.sec.3.5.lem.2.eq.3}
\frac{1}{d}\sum_{x\in\mathcal{Z}}\|x\|_1\leq{k-\kappa\!\left(B(d,p-1)\cap\mathbb{P}^d_\circ\right)-\frac{jp}{d}}\mbox{.}
\end{equation}

Hence, we obtain $\kappa(\mathcal{X})\leq{k}$ by combining (\ref{PPP.sec.3.5.lem.2.eq.0.5}), (\ref{PPP.sec.3.5.lem.2.eq.2}), and (\ref{PPP.sec.3.5.lem.2.eq.3}). Now, since all the elements of $\mathcal{Z}$ have $1$-norm $p$, and since has $\mathcal{S}$ has cardinality $j$,
\begin{equation}\label{PPP.sec.3.5.lem.2.eq.4}
|\mathcal{X}|=\left|B(d,p-1)\cap\mathbb{P}^d_\circ\right|+\frac{1}{p}\sum_{x\in\mathcal{Z}}\|x\|_1+j\mbox{.}
\end{equation}

Combining (\ref{PPP.sec.3.5.lem.2.eq.1}) with (\ref{PPP.sec.3.5.lem.2.eq.4}) provides the equality
$$
|\mathcal{X}|=\left|B(d,p-1)\cap\mathbb{P}^d_\circ\right|+\left\lfloor\frac{k-\kappa\!\left(B(d,p-1)\cap\mathbb{P}^d_\circ\right)}{p/d}\right\rfloor\!\mbox{,}
$$
whose right-hand side is precisely $\lfloor\lambda(d,k)\rfloor$.
\end{proof}

We now build a subset of $\mathbb{P}^d_\circ$ that achieves, when $d$ is greater than $2$ and a divisor of $p$, the upper bound on $\delta_z(d,k)$ provided by Lemmas \ref{PPP.sec.3.lem.0.5} and \ref{PPP.sec.3.lem.1}.

\begin{lem}\label{PPP.sec.3.5.lem.3}
Assume that $d$ is greater than $2$ and a divisor of $p$. If
\begin{itemize}
\item[(i)] $k-\kappa\!\left(B(d,p-1)\cap\mathbb{P}^d_\circ\right)$ is distinct from $p/d$,
\item[(ii)] $\kappa\!\left(B(d,p)\cap\mathbb{P}^d_\circ\right)-k$ is distinct from $p/d$,
\end{itemize}
then there exists a subset $\mathcal{X}$ of $\mathbb{P}^d_\circ$ satisfying $\kappa(\mathcal{X})\leq{k}$ and $|\mathcal{X}|=\lfloor\lambda(d,k)\rfloor$. Otherwise, there also exists a subset $\mathcal{X}$ of $\mathbb{P}^d_\circ$ such that $\kappa(\mathcal{X})\leq{k}$, but whose cardinality is $\lfloor\lambda(d,k)\rfloor-1$ instead of $\lfloor\lambda(d,k)\rfloor$.

\end{lem}
\begin{proof}
Assume that $d$ is a divisor of $p$. Consider the point $a$ of $\mathbb{Z}^d$ whose first coordinate is $p/d-1$, whose second coordinate is $p/d+1$, and whose all other coordinates are $p/d$. Further consider the point $b$ obtained from $a$ by exchanging the first two coordinates and the point $c$ obtained from $b$ by exchanging the second and third coordinates. As $d$ is greater than $2$, the three points $a$, $b$ and $c$ each admit two coordinates that are consecutive integers. Therefore, they are primitive. Note, in addition, that they have $1$-norm $p$. Consider the set
$$
\mathcal{Y}=(C\mathord{\cdot}a)\cup(C\mathord{\cdot}b)\cup(C\mathord{\cdot}c)\mbox{.}
$$

We are going to show that $\mathcal{Y}$ satisfies the requirements of Lemma \ref{PPP.sec.3.5.lem.2}. Observe that the orbits of $a$, $b$, and $c$ under the action of $C$ on $\mathbb{R}^d$ have cardinality $d$. Further observe that these orbits are pairwise disjoint if $d$ is greater than $3$. Therefore, in this case, $\mathcal{Y}$ has cardinality $3d$. If, however, $d$ is equal to $3$, then $\mathcal{Y}$ has cardinality only $6$ because  $C\mathord{\cdot}a$ coincides with $C\mathord{\cdot}b$. Further observe that all the points contained in $\mathcal{Y}$ have non-negative coordinates. Now, consider an integer $j$ such that $2\leq{j}\leq|\mathcal{Y}|-2$. If $j$ is even and at most $|\mathcal{Y}|/2$, denote
$$
\mathcal{S}_j=\left\{\tau^i(a):1\leq{i}\leq\frac{j}{2}\right\}\!\cup\!\left\{\tau^i(b):1\leq{i}\leq\frac{j}{2}\right\}\!\mbox{.}
$$

By construction, $\mathcal{S}_j$ is a subset of $j$ elements of $\mathcal{Y}$. Observe that $\mathcal{Y}$ is a subset of $[0,+\infty[^d$ and, therefore, so is $\mathcal{S}_j$. Further observe that all the coordinates of $a+b$ are equal to $2p/d$. Hence, all the coordinates of the sum of the points contained in $\mathcal{S}_j$ are equal to $jp/d$. As a consequence,
\begin{equation}\label{PPP.sec.3.5.lem.3.eq.1}
\kappa(\mathcal{S}_j)=\frac{jp}{d}\mbox{,}
\end{equation}

Now assume that $j$ is odd and at most $|\mathcal{Y}|/2$. We will distinguish two cases depending on whether $d$ is equal to $3$ or greater than $3$. If $d$ is equal to $3$, then $j$ must be equal to $3$, as this is the only odd number that is both greater than~$2$ and at most $|\mathcal{Y}|/2$. In this case, we denote
$$
\mathcal{S}_j=\left\{a,\tau(a),\tau^2(a)\right\}\mbox{,}
$$
and if $d$ is greater than $3$, we denote
$$
\mathcal{S}_j=\left\{\tau^i(a):1\leq{i}\leq\frac{j+1}{2}\right\}\!\cup\!\left\{\tau^i(b):1\leq{i}\leq\frac{j-3}{2}\right\}\!\cup\!\left\{\tau^{(j-1)/2}(c)\right\}\mbox{.}
$$

Again, $\mathcal{S}_j$ is a subset of $j$ elements of $\mathcal{Y}$. Moreover, all the coordinates of the sum of its elements are equal to $jp/d$. Hence $\mathcal{S}_j$ satisfies~(\ref{PPP.sec.3.5.lem.3.eq.1}).

In the case when $j$ is greater than $|\mathcal{Y}|/2$, denote
$$
\mathcal{S}_j=\mathcal{Y}\mathord{\setminus}\mathcal{S}_{|\mathcal{Y}|-j}\mbox{.}
$$

By construction, all the coordinates of the sum of the points contained in $\mathcal{Y}$ are equal to $|\mathcal{Y}|p/d$ and all the coordinates of the sum of the points contained in $\mathcal{S}_{|\mathcal{Y}|-j}$ to $(|\mathcal{Y}|-j)p/d$. Hence, $\mathcal{S}_j$ satisfies (\ref{PPP.sec.3.5.lem.3.eq.1}) in this case as well. As a consequence, when (\ref{PPP.sec.3.5.lem.2.eq.0}) holds, it follows from Lemma \ref{PPP.sec.3.5.lem.2} that there exists a subset of $\mathbb{P}^d_\circ$ of the desired cardinality, whose image by $\kappa$ is at most $k$.

In the remainder of the proof, we assume that (\ref{PPP.sec.3.5.lem.2.eq.0}) does not hold, and we review several cases depending on the values of $p/d$ and $k$. In each of this cases, we exhibit a subset of $\mathbb{P}^d_\circ$ with the desired properties. Recall that
$$
\lambda(d,k)=\left|B(d,p-1)\cap\mathbb{P}^d_\circ\right|+\frac{k-\kappa\!\left(B(d,p-1)\cap\mathbb{P}^d_\circ\right)}{p/d}\mbox{.}
$$

Observe that, when the difference $k-\kappa\!\left(B(d,p-1)\cap\mathbb{P}^d_\circ\right)$ is less than $p/d$ or, when that difference is exactly $p/d$ with $p>d$, then $B(d,p-1)\cap\mathcal{P}^d_\circ$ has the desired properties. Moreover, if $p$ is equal to $d$ and
$$
\frac{p}{d}\leq{k-\kappa\!\left(B(d,p-1)\cap\mathbb{P}^d_\circ\right)}<\frac{2p}{d}\mbox{,}
$$
then the union of $B(d,p-1)\cap\mathbb{P}^d_\circ$ with the singleton that contains the point of $\mathbb{R}^d$ whose all coordinates are equal to $1$ has the desired cardinality, and its image by $\kappa$ is at most $k$. Similarly, if $p$ is greater than $d$ and
$$
\frac{p}{d}<{k-\kappa\!\left(B(d,p-1)\cap\mathbb{P}^d_\circ\right)}<\frac{2p}{d}\mbox{,}
$$
then the set $\left(B(d,p-1)\cap\mathbb{P}^d_\circ\right)\cup\{a\}$, where $a$ is the point defined at the beginning of the proof is a subset of $\mathbb{P}^d_\circ$ with the desired properties.

Now, if $p$ is equal to $d$ and $\kappa\!\left(B(d,p)\cap\mathbb{P}^d_\circ\right)-k$ is at most $p/d$, then the set obtained from $B(d,p)\cap\mathbb{P}^d_\circ$ by removing the point whose all coordinates are equal to $1$ has the desired cardinality and its image by $\kappa$ is at most $k$. Similarly, if $p>d$ and $\kappa\!\left(B(d,p)\cap\mathbb{P}^d_\circ\right)-k$ is less than $p/d$, then removing from $B(d,p)\cap\mathbb{P}^d_\circ$ the point $a$ results in a subset of $\mathbb{P}^d_\circ$ of the right cardinality and image by $\kappa$. Finally, assume that $p$ is greater than $d$ and that $\kappa\!\left(B(d,p)\cap\mathbb{P}^d_\circ\right)-k$ is equal to $p/d$. In that case, $p/d$ is a divisor of $k-\kappa\!\left(B(d,p-1)\cap\mathbb{P}^d_\circ\right)$. Hence,
$$
\left\lfloor\frac{k-\kappa\!\left(B(d,p-1)\cap\mathbb{P}^d_\circ\right)}{p/d}\right\rfloor\!-1=\left\lfloor\frac{k-1-\kappa\!\left(B(d,p-1)\cap\mathbb{P}^d_\circ\right)}{p/d}\right\rfloor\!\mbox{.}
$$

Note in particular that $k-1-\kappa\!\left(B(d,p-1)\cap\mathbb{P}^d_\circ\right)$ cannot be a multiple of $p/d$. Therefore the set that we have already built, of cardinality
$$
\left|B(d,p-1)\cap\mathbb{P}^d_\circ\right|+\left\lfloor\frac{k-1-\kappa\!\left(B(d,p-1)\cap\mathbb{P}^d_\circ\right)}{p/d}\right\rfloor\!
$$
whose image by $\kappa$ is at most $k-1$ has the desired properties.
\end{proof}

Assume that $d$ is greater than $2$ and consider the set
$$
\mathbb{E}=\bigcup\left\{\kappa\!\left(B(d,p-1)\cap\mathbb{P}^d_\circ\right)+\frac{p}{d},\kappa\!\left(B(d,p)\cap\mathbb{P}^d_\circ\right)-\frac{p}{d}\right\}\mbox{,}
$$
where the union is over the positive integers $p$ that admit $d$ as a proper divisor. Note in particular that $\mathbb{E}$ depends on the dimension. With this notation Theorem \ref{PPP.sec.0.thm.2} is a consequence, for all dimensions above $2$, of Lemmas \ref{PPP.sec.3.lem.1}, \ref{PPP.sec.3.5.lem.1}, and \ref{PPP.sec.3.5.lem.3}, except for the part on the density of $\mathbb{E}$ within $\mathbb{N}$.

The following result completes the proof of Theorem \ref{PPP.sec.0.thm.2} when $d>2$.

\begin{lem}\label{PPP.sec.3.5.lem.4}
If $d>2$, then $\displaystyle\lim_{k\rightarrow\infty}\frac{\left|\mathbb{E}\cap[1,k]\right|}{k^{1/(d+1)}}=\frac{\left[4(d+1)!\zeta(d)\right]^{1/(d+1)}}{d}$.
\end{lem}
\begin{proof}
Consider the smallest integer $p$ such that $k<\kappa\!\left(B(d,p)\cap\mathbb{P}^d_\circ\right)$ and observe that the cardinality of $\mathbb{E}\cap[1,k]$ is at most twice the number of multiples of $d$ less than or equal to $p$ and at least that number minus $4$. Therefore,
$$
\frac{2p/d-4}{\!\left[\kappa\!\left(B(d,p)\cap\mathbb{P}^d_\circ\right)\right]^{1/(d+1)}}\leq\frac{\left|\mathbb{E}\cap[1,k]\right|}{k^{1/(d+1)}}\leq\frac{2p/d}{\!\left[\kappa\!\left(B(d,p-1)\cap\mathbb{P}^d_\circ\right)\right]^{1/(d+1)}}\mbox{.}
$$

However, according to Theorem 4.2 from \cite{DezaPourninSukegawa2020},
$$
\lim_{p\rightarrow\infty}\frac{\left[\kappa\!\left(B(d,p)\cap\mathbb{P}^d_\circ\right)\right]^{1/(d+1)}}{p}=\left(\frac{2^{d-1}}{(d+1)!\zeta(d)}\right)^{1/(d+1)}\mbox{,}
$$
and the desired result follows.
\end{proof}

Note that the asymptotic estimate of the density of $\mathbb{E}$ within $\mathbb{N}$ provided by Lemma~\ref{PPP.sec.3.5.lem.4} is exact. This estimate is two orders lower than the upper bound in the statement of Theorem \ref{PPP.sec.0.thm.2}.
Indeed, we shall see in the next section that $\mathbb{E}$ is more dense in $2$-dimensions. However, if the statement of Theorem \ref{PPP.sec.0.thm.2} is restricted to dimensions above $2$, the upper bound on the density of $\mathbb{E}$ can be replaced by the exact estimate from Lemma~\ref{PPP.sec.3.5.lem.4}.

\section{An expression for $\delta_z(2,k)$}\label{PPP.sec.4}

This section is devoted to proving the $2$-dimensional case of Theorem \ref{PPP.sec.0.thm.2}. We will assume throughout the section that $d$ is equal to $2$, that $k$ is fixed, and that $p$ denotes the smallest integer such that $k$ is less than $\kappa\!\left(B(2,p)\cap\mathbb{P}^d_\circ\right)$. Observe that, when $p$ is odd, the value of $\delta_z(2,k)$ is provided by Lemmas~\ref{PPP.sec.3.lem.0.5} and~\ref{PPP.sec.3.5.lem.1}. Therefore, we only have to treat the case when $p$ is even. Moreover, note that Lemmas \ref{PPP.sec.3.lem.0.5}, \ref{PPP.sec.3.lem.1}, and \ref{PPP.sec.3.lem.2} provide the desired upper bound on $\delta_z(2,k)$ for all $k$. Hence, we only need to build subsets of $\mathbb{P}^2_\circ$ that achieve that bound for even $p$. We first consider the case when $p$ is a multiple of $4$.

\begin{lem}\label{PPP.sec.4.lem.1}
Assume that $p$ is a multiple of $4$. If $k-\kappa\!\left(B(2,p-1)\cap\mathbb{P}^2_\circ\right)$ is not a multiple of $p/2$ by an odd integer, then there exists a subset $\mathcal{X}$ of $\mathbb{P}^2_\circ$ satisfying $\kappa(\mathcal{X})\leq{k}$ and $|\mathcal{X}|=\lfloor\lambda(2,k)\rfloor$. Otherwise, there also exists such a subset of $\mathbb{P}^2_\circ$, but whose cardinality is $\lfloor\lambda(2,k)\rfloor-1$ instead of $\lfloor\lambda(2,k)\rfloor$.
\end{lem}
\begin{proof}
Consider the point $a$ such that $a_1=p/2-1$ and $a_2=p/2+1$. Note that, since $p/2$ is even, $a$ is primitive. Moreover, $a$ has $1$-norm $p$.

Observe that the set of the points of $1$-norm $p$ from $\mathbb{P}^2_\circ$ is partitioned by the orbits of its elements under the action of $C$ on $\mathbb{R}^2$. Moreover, since $p$ is greater than $2$ the orbit of any of these elements under the action of $C$ on $\mathbb{R}^2$ has cardinality exactly $2$. Therefore, one can take the union of some of these orbits in order to build a set $\mathcal{Z}$ of points of $1$-norm $p$ from $\mathbb{P}^2_\circ$ that is invariant under the action of $C$ on $\mathbb{R}^2$, that is disjoint from $C\mathord{\cdot}a$ and that satisfies
\begin{equation}\label{PPP.sec.4.lem.1.eq.1}
\frac{k-\kappa\!\left(B(2,p-1)\cap\mathbb{P}^2_\circ\right)}{p/2}-2<\frac{1}{p}\sum_{x\in\mathcal{Z}}\|x\|_1\leq\frac{k-\kappa\!\left(B(2,p-1)\cap\mathbb{P}^2_\circ\right)}{p/2}\mbox{.}
\end{equation}

Consider the difference
\begin{equation}\label{PPP.sec.4.lem.1.eq.2}
g=\frac{k-\kappa\!\left(B(2,p-1)\cap\mathbb{P}^2_\circ\right)}{p/2}-\frac{1}{p}\sum_{x\in\mathcal{Z}}\|x\|_1\mbox{.}
\end{equation}

It follows from (\ref{PPP.sec.4.lem.1.eq.1}) that $0\leq{g}<2$. If $g\leq1$, consider the set
$$
\mathcal{X}=\left(B(2,p-1)\cap\mathbb{P}^2_\circ\right)\cup\mathcal{Z}\mbox{.}
$$

By construction, $\mathcal{X}$ is invariant under the action of $C$ on $\mathbb{R}^2$. In addition, as the elements of $\mathcal{Z}$ all have $1$-norm $p$, it follows from Proposition \ref{PPP.sec.3.prop.1} that
\begin{equation}\label{PPP.sec.4.lem.1.eq.3}
|\mathcal{X}|=\kappa\!\left(B(2,p-1)\cap\mathbb{P}^2_\circ\right)+\frac{1}{2}\sum_{x\in\mathcal{Z}}\|x\|_1\mbox{.}
\end{equation}

However, as $g\geq0$, it follows from (\ref{PPP.sec.4.lem.1.eq.2}) that
$$
\frac{1}{2}\sum_{x\in\mathcal{Z}}\|x\|_1\leq{k-\kappa\!\left(B(2,p-1)\cap\mathbb{P}^2_\circ\right)}\mbox{.}
$$

Combining this with (\ref{PPP.sec.4.lem.1.eq.3}) shows that $\kappa(Z)\leq{k}$. Now recall that
$$
\lambda(2,k)=\left|B(2,p-1)\cap\mathbb{P}^2_\circ\right|+\frac{k-\kappa\!\left(B(2,p-1)\cap\mathbb{P}^2_\circ\right)}{p/2}\mbox{.}
$$

Since all the elements of $\mathcal{Z}$ all have $1$-norm $p$,
\begin{equation}\label{PPP.sec.4.lem.1.eq.3.5}
|\mathcal{X}|=\left|B(2,p-1)\cap\mathbb{P}^2_\circ\right|+\frac{1}{p}\sum_{x\in\mathcal{Z}}\|x\|_1\mbox{.}
\end{equation}

As $g$ is non-negative, combining (\ref{PPP.sec.4.lem.1.eq.2}) with (\ref{PPP.sec.4.lem.1.eq.3.5}) shows that $|\mathcal{X}|=\lfloor\lambda(d,k)\rfloor$. Moreover, when $g$ is equal to $1$, we obtain that $\mathcal{X}$ has cardinality $\lfloor\lambda(d,k)\rfloor-1$. As $g$ is equal to $1$ if and only if $k-\kappa\!\left(B(2,p-1)\cap\mathbb{P}^2_\circ\right)$ is an odd multiple of $p/2$, the lemma is proven in this case.

Now assume that $g$ is greater than $1$ and denote
$$
\mathcal{X}=\left(B(2,p-1)\cap\mathbb{P}^2_\circ\right)\cup\mathcal{Z}\cup\{a\}\mbox{,}
$$

According to Proposition \ref{PPP.sec.3.prop.1},
\begin{equation}\label{PPP.sec.4.lem.1.eq.4}
\kappa(\mathcal{X})=\kappa\!\left(B(2,p-1)\cap\mathbb{P}^2_\circ\right)+\frac{1}{2}\sum_{x\in\mathcal{Z}}\|x\|_1+\frac{p}{2}+1\mbox{.}
\end{equation}

As $g$ is greater than $1$, (\ref{PPP.sec.4.lem.1.eq.2}) yields
$$
\frac{1}{2}\sum_{x\in\mathcal{X}}\|x\|_1<k-\kappa\!\left(B(2,p-1)\cap\mathbb{P}^2_\circ\right)-\frac{p}{2}\mbox{.}
$$

Since this inequality is strict, combining it with (\ref{PPP.sec.4.lem.1.eq.4}) shows that $\kappa(Z)\leq{k}$. Finally, as all the elements of $\mathcal{Z}$ have $1$-norm $p$,
$$
|\mathcal{X}|=\left|B(2,p-1)\cap\mathbb{P}^2_\circ\right|+\frac{1}{p}\sum_{x\in\mathcal{Z}}\|x\|_1+1\mbox{.}
$$

As $g$ greater than $1$, combining this with (\ref{PPP.sec.4.lem.1.eq.2}) shows that $|\mathcal{X}|=\lfloor\lambda(d,k)\rfloor$.
\end{proof}

Let us now treat the case when $p$ is even and $p/2$ is odd. The sub-cases when
$$
k\in\left\{\kappa\!\left(B(2,p-1)\cap\mathbb{P}^2_\circ\right)+\frac{p}{2}+1,\kappa\!\left(B(2,p-1)\cap\mathbb{P}^2_\circ\right)-\frac{p}{2}+1\right\}
$$
turn out to be particularly interesting. In these two sub-cases, the subsets $\mathcal{X}$ of $\mathbb{P}^2_\circ$ such that $|\mathcal{X}|=\delta_z(2,k)$ and $\kappa(\mathcal{X})\leq{k}$ cannot be between $B(2,p-1)\cap\mathbb{P}^2_\circ$ and $B(2,p)\cap\mathbb{P}^2_\circ$ in terms of inclusion. Let us begin with the case when $k$ is equal to $\kappa\!\left(B(2,p-1)\cap\mathbb{P}^2_\circ\right)+p/2+1$.

\begin{prop}\label{PPP.sec.4.prop.1}
If $p$ is even, $p/2$ is odd, and $p>2$, then there exists a subset $\mathcal{X}$ of $\mathbb{P}^2_\circ$ of cardinality $\left|B(2,p-1)\cap\mathbb{P}^2_\circ\right|+1$ such that
$$
\kappa(\mathcal{X})\leq\kappa\!\left(B(2,p-1)\cap\mathbb{P}^2_\circ\right)+\frac{p}{2}+1\mbox{.}
$$

In addition, for any such subset $\mathcal{X}$ of $\mathbb{P}^2_\circ$, either
\begin{itemize}
\item[(i)] some point in $B(2,p-1)\cap\mathbb{P}^2_\circ$ is not contained in $\mathcal{X}$, or
\item[(ii)] some point in $\mathcal{X}$ is not contained in $B(2,p)\cap\mathbb{P}^2_\circ$.
\end{itemize}
\end{prop}
\begin{proof}
Let $a$ be the point in $\mathbb{Z}^2_\circ$ such that $a_1=p/2$ and $a_2=p/2+1$. Note that $a$ belongs to $\mathbb{P}^2_\circ$ and consider the set
$$
\mathcal{X}=\left(B(2,p-1)\cap\mathbb{P}^2_\circ\right)\cup\{a\}\mbox{.}
$$

Note that $|\mathcal{X}|=\left|B(2,p-1)\cap\mathbb{P}^2_\circ\right|+1$. Recall that $B(2,p-1)\cap\mathbb{P}^2_\circ$ is invariant under the action of $C$ on $\mathbb{R}^2$. As in addition, $\max\{|a_1|,|a_2|\}=p/2+1$, 
$$
\kappa(\mathcal{X})\leq\kappa\!\left(B(2,p-1)\cap\mathbb{P}^2_\circ\right)+\frac{p}{2}+1\mbox{,}
$$
as desired. Now observe that, if a subset $\mathcal{X}$ of $\mathbb{P}^2_\circ$ satisfies that inequality and admits $B(2,p-1)\cap\mathbb{P}^2_\circ$ as a subset then, by Proposition \ref{PPP.sec.3.prop.1}, the only point $z$ of $1$-norm $p$ contained in $\mathcal{X}$ must satisfy
\begin{equation}\label{PPP.sec.4.prop.1.eq.1}
\max\{|z_1|,|z_2|\}\leq\frac{p}{2}+1\mbox{.}
\end{equation}

If in addition, $z$ is has $1$-norm $p$, the two coordinates of $z$ cannot have the same absolute value because $z$ is primitive and $p$ is greater than $2$. Hence, according to (\ref{PPP.sec.4.prop.1.eq.1}), the absolute value of the coordinates of $z$ must be $p/2-1$ and $p/2+1$. However, as $p/2$ is odd, these numbers are both even. As a consequence, $z$ cannot be primitive, a contradiction.
\end{proof}

The argument, for the proof of the following proposition is similar to that in the proof of Proposition \ref{PPP.sec.4.prop.1}, and we only sketch it.

\begin{prop}\label{PPP.sec.4.prop.2}
If $p$ is even, $p/2$ is odd, and $p>2$, then there exists a subset $\mathcal{X}$ of $\mathbb{P}^2_\circ$ of cardinality $\left|B(2,p)\cap\mathbb{P}^2_\circ\right|-1$ such that
$$
\kappa(\mathcal{X})\leq\kappa\!\left(B(2,p)\cap\mathbb{P}^2_\circ\right)-\frac{p}{2}+1\mbox{.}
$$

In addition, for any such subset $\mathcal{X}$ of $\mathbb{P}^2_\circ$, either
\begin{itemize}
\item[(i)] some point in $B(2,p-1)\cap\mathbb{P}^2_\circ$ is not contained in $\mathcal{X}$, or
\item[(ii)] some point in $\mathcal{X}$ is not contained in $B(2,p)\cap\mathbb{P}^2_\circ$.
\end{itemize}
\end{prop}
\begin{proof}
Consider the set $\mathcal{X}$ obtained by removing from the intersection of $B(2,p)$ and $\mathbb{P}^2_\circ$ the point whose first coordinate is $p/2-1$ and whose second coordinate is $p/2$. This set has cardinality $|B(2,p)\cap\mathbb{P}^2_\circ|-1$ and satisfies
\begin{equation}\label{PPP.sec.4.prop.2.eq.1}
\kappa(\mathcal{X})\leq\kappa\!\left(B(2,p)\cap\mathbb{P}^2_\circ\right)-\frac{p}{2}+1\mbox{,}
\end{equation}
as desired. Moreover, using an argument similar to that of the proof of Proposition \ref{PPP.sec.4.prop.1}, we show that any subset $\mathcal{X}$ of $B(2,p)\cap\mathbb{P}^2_\circ$ that satisfies (\ref{PPP.sec.4.prop.2.eq.1}), has cardinality $|B(2,p)\cap\mathbb{P}^2_\circ|-1$, and admits $B(2,p-1)\cap\mathbb{P}^2_\circ$ as a subset must miss a point from $B(2,p)\cap\mathbb{P}^2_\circ$ whose absolute value of the coordinates are $p/2-1$ and $p/2+1$. However, as $p/2$ is odd, this point cannot be primitive.
\end{proof}

The following lemma is a consequence of the two above propositions.

\begin{lem}\label{PPP.sec.4.lem.2}
Assume that $p$ is even and that $p/2$ is odd. Further assume that either $k-\kappa\!\left(B(2,p-1)\cap\mathbb{P}^2_\circ\right)<p$ or $\kappa\!\left(B(2,p)\cap\mathbb{P}^2_\circ\right)-k\leq{p/2}$. If
\begin{itemize}
\item[(i)] $k-\kappa\!\left(B(2,p-1)\cap\mathbb{P}^2_\circ\right)$ is distinct from $p/2$,
\item[(ii)] $\kappa\!\left(B(2,p)\cap\mathbb{P}^2_\circ\right)-k$ is distinct from $p/2$,
\end{itemize}
then there exists a subset $\mathcal{X}$ of $\mathbb{P}^2_\circ$ satisfying $\kappa(\mathcal{X})\leq{k}$ and $|\mathcal{X}|=\lfloor\lambda(2,k)\rfloor$. Otherwise, there also exists a subset $\mathcal{X}$ of $\mathbb{P}^2_\circ$ such that $\kappa(\mathcal{X})\leq{k}$, but whose cardinality is $\lfloor\lambda(2,k)\rfloor-1$ instead of $\lfloor\lambda(2,k)\rfloor$.
\end{lem}
\begin{proof}
Assume that $p$ is even and $p/2$ is odd and recall that
$$
\lambda(2,k)=\left|B(2,p-1)\cap\mathbb{P}^2_\circ\right|+\frac{k-\kappa\!\left(B(2,p-1)\cap\mathbb{P}^2_\circ\right)}{p/2}\mbox{.}
$$

Let us first treat the case when
$$
k-\kappa\!\left(B(2,p-1)\cap\mathbb{P}^2_\circ\right)<p\mbox{.}
$$

Observe that, if $k$ is less than $\kappa\!\left(B(2,p-1)\cap\mathbb{P}^2_\circ\right)+p/2$ or $k$ is equal to this quantity but $p$ is greater than $2$, then we can take $\mathcal{X}=B(2,p-1)\cap\mathbb{P}^2_\circ$.

If $p=2$ and $k-\kappa\!\left(B(2,p-1)\cap\mathbb{P}^2_\circ\right)\geq{p/2}$, observe that
\begin{equation}\label{PPP.sec.4.lem.2.eq.1}
\left\lfloor\frac{k-\kappa\!\left(B(2,p-1)\cap\mathbb{P}^2_\circ\right)}{p/2}\right\rfloor\!=1\mbox{.}
\end{equation}

Let $\mathcal{X}$ be the union of $B(2,p-1)\cap\mathbb{P}^2_\circ$ with the singleton that contains the point whose two coordinates are equal to $1$. This set has cardinality $\left|B(2,p-1)\cap\mathbb{P}^2_\circ\right|+1$ and satisfies $\kappa(\mathcal{X})\leq\kappa\!\left(B(2,p-1)\cap\mathbb{P}^2_\circ\right)+p/2$, as desired. Therefore, the lemma holds in this case.

If $p>2$ and $k-\kappa\!\left(B(2,p-1)\cap\mathbb{P}^2_\circ\right)>p/2$, then (\ref{PPP.sec.4.lem.2.eq.1}) also holds. Therefore, a set $\mathcal{X}$ with the desired properties is provided by Proposition \ref{PPP.sec.4.prop.1}.

Now assume that 
$$
\kappa\!\left(B(2,p)\cap\mathbb{P}^2_\circ\right)-k\leq\frac{p}{2}\mbox{.}
$$

If $p$ is equal to $2$, then consider the set $\mathcal{X}$ obtained from $B(2,p)\cap\mathbb{P}^2_\circ$ by removing the point whose two coordinates are equal to $1$. This set has cardinality $\left|B(2,p)\cap\mathbb{P}^2_\circ\right|-1$, as desired, and satisfies $\kappa(\mathcal{X})\leq\kappa\!\left(B(2,p)\cap\mathbb{P}^2_\circ\right)-p/2$. Therefore, the lemma holds in this case.

If $p>2$ and $\kappa\!\left(B(2,p)\cap\mathbb{P}^2_\circ\right)-k<p/2$, a set $\mathcal{X}$ with the desired properties is provided by Proposition \ref{PPP.sec.4.prop.2}, and if $\kappa\!\left(B(2,p)\cap\mathbb{P}^2_\circ\right)-k=p/2$, then consider any point $x$ of $1$-norm $p$ in $B(2,p)\cap\mathbb{P}^2_\circ$. Any set obtained by removing the two elements of $C\mathord{\cdot}x$ from $B(2,p)\cap\mathbb{P}^2_\circ$ will have the desired properties.
\end{proof}

%
%
We are ready to build the announced subsets of $\mathbb{P}^2_\circ$. The construction relies on Lemma \ref{PPP.sec.4.lem.2} when $k$ is close to $\kappa\!\left(B(2,p-1)\cap\mathbb{P}^2_\circ\right)$ or to $\kappa\!\left(B(2,p)\cap\mathbb{P}^2_\circ\right)$, and on Lemma \ref{PPP.sec.3.5.lem.2} otherwise. We also use the property that, if $n$ is an odd integer and $m$ a positive integer, then $n-2^m$ and $n+2^m$ are relatively prime.

\begin{lem}\label{PPP.sec.4.lem.3}
Assume that $p$ is even and that $p/2$ is odd. If
\begin{itemize}
\item[(i)] $k-\kappa\!\left(B(2,p-1)\cap\mathbb{P}^2_\circ\right)$ is distinct from $p/2$,
\item[(ii)] $\kappa\!\left(B(2,p)\cap\mathbb{P}^2_\circ\right)-k$ is distinct from $p/2$,
\end{itemize}
then there exists a subset $\mathcal{X}$ of $\mathbb{P}^2_\circ$ satisfying $\kappa(\mathcal{X})\leq{k}$ and $|\mathcal{X}|=\lfloor\lambda(2,k)\rfloor$. Otherwise, there also exists a subset $\mathcal{X}$ of $\mathbb{P}^2_\circ$ such that $\kappa(\mathcal{X})\leq{k}$, but whose cardinality is $\lfloor\lambda(2,k)\rfloor-1$ instead of $\lfloor\lambda(2,k)\rfloor$.
\end{lem}
\begin{proof}
Assume that $p$ is even and that $p/2$ is odd. We will split the proof into three cases. The first case, when $k-\kappa\!\left(B(2,p-1)\cap\mathbb{P}^2_\circ\right)$ is less than $p$ or $\kappa\!\left(B(2,p)\cap\mathbb{P}^2_\circ\right)-k$ is at most $p/2$, is taken care of by Lemma \ref{PPP.sec.4.lem.2}.

Now assume that
\begin{equation}\label{PPP.sec.4.lem.3.eq.1}
\kappa\!\left(B(2,p-1)\cap\mathbb{P}^2_\circ\right)+p\leq{k}<\kappa\!\left(B(2,p)\cap\mathbb{P}^2_\circ\right)-\frac{p}{2}\mbox{,}
\end{equation}
and that $p$ is at least $10$. In this case, we are going to use Lemma \ref{PPP.sec.3.5.lem.2}. Consider the point $a$ of $\mathbb{R}^2$ whose first coordinate is $p/2-2$ and whose second coordinate is $p/2+2$. Further consider the point $b$ obtained from $a$ by negating the second coordinate and the point $c$ whose first coordinate is $p/2+4$ and whose second coordinate is $p/2-4$. These points all have $1$-norm $p$ and, as $p/2$ is odd, they are primitive. Since the orbits of these points under the action of $C$ on $\mathbb{R}^2$ have cardinality $2$, and since these orbits are pairwise disjoint, the set
$$
\mathcal{Y}=(C\mathord{\cdot}a)\cup(C\mathord{\cdot}b)\cup(C\mathord{\cdot}c)\mbox{.}
$$
has cardinality $6$. Let us consider an integer $j$ such that $2\leq{j}\leq4$. We will exhibit a subset of $j$ elements of $\mathcal{S}$ that satisfies (\ref{PPP.sec.3.5.lem.2.eq.0.5}), as required by Lemma~\ref{PPP.sec.3.5.lem.2}. If $j$ is equal to $2$ we take for $\mathcal{S}$ the set $\{a,\tau(a)\}$. If $j$ is equal to $3$, we take $\mathcal{S}=\{a,b,c\}$, and if $j$ is equal to $4$, we take $\mathcal{S}=\{a,\tau(a),b,\tau(b)\}$. By the choice of $a$, $b$, and $c$, $\mathcal{S}$ satisfies (\ref{PPP.sec.3.5.lem.2.eq.0.5}). Therefore, according to Lemma \ref{PPP.sec.3.5.lem.2}, there exists a subset $\mathbb{P}^2_\circ$ such that $\kappa(\mathcal{X})\leq{k}$ and $|\mathcal{X}|=\lfloor\lambda(2,k)\rfloor$, as desired.

Now assume that (\ref{PPP.sec.4.lem.3.eq.1}) still holds, but that $p=6$. Recall that
$$
\lambda(2,k)=\left|B(2,p-1)\cap\mathbb{P}^2_\circ\right|+\frac{k-\kappa\!\left(B(2,p-1)\cap\mathbb{P}^2_\circ\right)}{p/2}\mbox{.}
$$

It follows from Theorem \ref{PPP.sec.0.thm.1} that
$$
\kappa\!\left(B(2,6)\cap\mathbb{P}^2_\circ\right)-\kappa\!\left(B(2,5)\cap\mathbb{P}^2_\circ\right)=12\mbox{.}
$$

Hence, according to (\ref{PPP.sec.4.lem.3.eq.1}), $6\leq{k-\kappa\!\left(B(2,5)\cap\mathbb{P}^2_\circ\right)}\leq8$. In this range,
$$
\left\lfloor\frac{k-\kappa\!\left(B(2,5)\cap\mathbb{P}^2_\circ\right)}{6/2}\right\rfloor\!=2\mbox{.}
$$

Therefore, the subset $\mathcal{X}$ of $\mathbb{P}^2_\circ$ made up of the points from $B(2,5)\cap\mathbb{P}^2_\circ$ together with the two primitive points of $1$-norm $6$ whose coordinates are $1$ and $5$ satisfies the desired properties as it has cardinality $\left|B(2,5)\cap\mathbb{P}^2_\circ\right|+2$ and
$$
\kappa(\mathcal{X})\leq\kappa\!\left(B(2,5)\cap\mathbb{P}^2_\circ\right)+6\mbox{.}
$$

Finally, note that no other case needs to be treated. In particular, when $p$ is equal to $2$, it follows from Theorem \ref{PPP.sec.0.thm.1} that $\kappa\!\left(B(2,p-1)\cap\mathbb{P}^2_\circ\right)=1$ and $\kappa\!\left(B(2,p)\cap\mathbb{P}^2_\circ\right)=3$. Therefore, in this case, (\ref{PPP.sec.4.lem.3.eq.1}) is never satisfied.
\end{proof}

For any positive integer $p$, let $I_p$ be the set of the odd multiples of $p/2$ that lie between $\kappa\!\left(B(2,p-1)\cap\mathbb{P}^2_\circ\right)$ and $\kappa\!\left(B(2,p)\cap\mathbb{P}^2_\circ\right)$. Denote
$$
\mathbb{E}_e=\bigcup_{p=1}^\infty{I_{4p}}\mbox{.}
$$

Further consider the set
$$
\mathbb{E}_o=\bigcup\left\{\kappa\!\left(B(2,p-1)\cap\mathbb{P}^2_\circ\right)+\frac{p}{2},\kappa\!\left(B(2,p)\cap\mathbb{P}^2_\circ\right)-\frac{p}{2}\right\}\!\mbox{,}
$$
where the union ranges over the even numbers $p$ greater than $2$ such that $p/2$ is odd. Denote $\mathbb{E}=\mathbb{E}_e\cup\mathbb{E}_o$. With this notation, the expression for $\delta_z(2,k)$ provided by Theorem \ref{PPP.sec.0.thm.2} is a consequence of Lemmas \ref{PPP.sec.3.lem.0.5}, \ref{PPP.sec.3.lem.1}, \ref{PPP.sec.3.lem.2}, \ref{PPP.sec.3.5.lem.1}, \ref{PPP.sec.4.lem.1}, and \ref{PPP.sec.4.lem.3}. We now estimate the density of $\mathbb{E}$ within $\mathbb{N}$ in the $2$-dimensional case, which completes the proof of Theorem \ref{PPP.sec.0.thm.2}.

\begin{lem}\label{PPP.sec.4.lem.4}
If $d$ is equal to $2$, then $\displaystyle\lim_{k\rightarrow\infty}\frac{\left|\mathbb{E}\cap[1,k]\right|}{k}=0$.
\end{lem}
\begin{proof}
Observe that there are exactly
$$
\frac{\kappa\!\left(B(2,i)\cap\mathbb{P}^2_\circ\right)-\kappa\!\left(B(2,i-1)\cap\mathbb{P}^2_\circ\right)}{i}
$$
odd multiples of $i/2$ between $\kappa\!\left(B(2,i-1)\cap\mathbb{P}^2_\circ\right)$ and $\kappa\!\left(B(2,i)\cap\mathbb{P}^2_\circ\right)$. According to Theorem \ref{PPP.sec.0.thm.1}, this quantity is precisely $J_1(i)$. Hence, when $d=2$,
$$
\left|\mathbb{E}\cap[1,k]\right|\leq\sum_{i=1}^pJ_1(i)\mbox{,}
$$
where $p$ is the smallest integer such that $k<\kappa\!\left(B(2,p)\cap\mathbb{P}^2_\circ\right)$. Therefore,
\begin{equation}\label{PPP.sec.4.lem.4.eq.1}
\frac{\left|\mathbb{E}\cap[1,k]\right|}{k}\leq\frac{1}{\kappa\!\left(B(2,p-1)\cap\mathbb{P}^2_\circ\right)}\sum_{i=1}^pJ_1(i)\mbox{.}
\end{equation}

Now recall that $J_1(i)$ is Euler's totient function (see, for instance \cite{HardyWright1938}). In particular, we have the following estimate.
$$
\lim_{p\rightarrow\infty}\frac{1}{p^2}\sum_{i=1}^pJ_1(i)=\frac{3}{\pi^2}\mbox{.}
$$

In addition, according to Theorem 4.2 from \cite{DezaPourninSukegawa2020},
$$
\lim_{p\rightarrow\infty}\frac{\kappa\!\left(B(2,p)\cap\mathbb{P}^2_\circ\right)}{p^3}=\frac{1}{3\zeta(2)}\mbox{.}
$$

Hence letting $k$ go to infinity in (\ref{PPP.sec.4.lem.4.eq.1}) proves the lemma.
\end{proof}

\section{The geometry of packed primitive point sets}\label{PPP.sec.5}

In this section, we gather results on which sets of primitive points solve our packing problem. The first one is Theorem \ref{PPP.sec.0.thm.4}, announced in the introduction, that we prove using the constructions in the previous two sections.

\begin{proof}[Proof of Theorem \ref{PPP.sec.0.thm.4}]
First observe that all the subsets $\mathcal{X}$ of $\mathbb{P}^d_\circ$ such that $|\mathcal{X}|=\delta_z(d,k)$ and $\kappa(\mathcal{X})\leq{k}$ constructed in the proofs of Lemmas \ref{PPP.sec.3.5.lem.1},  \ref{PPP.sec.3.5.lem.2}, and \ref{PPP.sec.3.5.lem.3} satisfy the double inequality
\begin{equation}\label{PPP.sec.0.thm.4.eq.1}
B(d,p-1)\cap\mathbb{P}^d_\circ\subset\mathcal{X}\subset{B(d,p)\cap\mathbb{P}^d_\circ}\mbox{,}
\end{equation}
where $p$ is the smallest integer such that $k$ is less than $\kappa\!\left(B(d,p)\cap\mathbb{P}^d_\circ\right)$. Hence, the theorem holds when $d>2$, and when $d=2$ while $p$ is odd.

Now assume that $d=2$ and that $p$ is even. Observe that the subset $\mathcal{X}$ of $\mathbb{P}^d_\circ$ such that $|\mathcal{X}|=\delta_z(d,k)$ and $\kappa(\mathcal{X})\leq{k}$ constructed in the proof of Lemma \ref{PPP.sec.4.lem.1} also satisfies (\ref{PPP.sec.0.thm.4.eq.1}). The same holds for the construction given in the proof of Lemma \ref{PPP.sec.4.lem.3} except when $p/2$ is odd, $p$ is greater than $2$ and either
$$
k=\kappa\!\left(B(2,p-1)\cap\mathbb{P}^d_\circ\right)+\frac{p}{2}+1\mbox{,}
$$
or
$$
k=\kappa\!\left(B(2,p)\cap\mathbb{P}^d_\circ\right)-\frac{p}{2}+1\mbox{,}
$$
the cases treated by Propositions \ref{PPP.sec.4.prop.1} and \ref{PPP.sec.4.prop.2}. According to these propositions, a subset $\mathcal{X}$ of $\mathbb{P}^d_\circ$ such that $|\mathcal{X}|=\delta_z(d,k)$ and $\kappa(\mathcal{X})\leq{k}$ cannot satisfy (\ref{PPP.sec.0.thm.4.eq.1}) in these cases, and the theorem follows.
\end{proof}

Let us turn our attention to the unicity of the subsets of $\mathbb{P}^d_\circ$ that solve our primitive point packing problem. Recall that we need to prove that unicity is not achieved apart for the cases already identified in \cite{DezaPourninSukegawa2020}. It is noteworthy that this will not require any knowledge on the value of $\delta_z(d,k)$. Indeed, the general idea in order to prove that several subsets of $\mathbb{P}^d_\circ$ solve our packing problem, is to consider such a subset, and to replace some of the points it contains without affecting its cardinality or increasing its image by $\kappa$.

We will do this by means of the following proposition.

\begin{prop}\label{PPP.sec.5.prop.1}
For any point $z$ in $\mathbb{P}^d_\circ$, there exist a set $\mathcal{K}$ of $d$ points from $\mathbb{P}^d_\circ$, all of $1$-norm $\|z\|_1$, and a partition of $\mathcal{K}$ into $d/|C\mathord{\cdot}z|$ subsets, each of cardinality $|C\mathord{\cdot}z|$ such that, if $\mathcal{L}$ is any of these subsets, then
$$
\kappa(C\mathord{\cdot}z)=\sum_{x\in\mathcal{L}}|x_i|
$$
for every integer $i$ satisfying $1\leq{i}\leq{d}$. In addition, there exists a point of $\mathbb{P}^d_\circ$ whose orbit under the action of $C$ on $\mathbb{R}^d$ admits $\mathcal{L}$ as a subset.
\end{prop}
\begin{proof}
Consider a point $z$ in $\mathbb{P}^d_\circ$ and note that when $C\mathord{\cdot}z$ has cardinality $d$ then, we can take $\mathcal{K}=C\mathord{\cdot}z$. Assume that $|C\mathord{\cdot}z|$ is less than $d$. We can require without loss of generality that the first coordinate of $z$ is non-zero by exchanging it for a point from $C\mathord{\cdot}z$ with that property. Pick an integer $j$ such that
$$
2\leq{j}\leq\frac{d}{|C\mathord{\cdot}z|}\mbox{.}
$$

Consider the point $y$ such that $y_i=-z_i$ when $i$ is equal to $(j-1)|C\mathord{\cdot}z|$ and $y_i=z_i$ otherwise. Denote
$$
\mathcal{L}_j=\left\{\tau^i(y):0\leq{i}<|C\mathord{\cdot}z|\right\}\!\mbox{.}
$$

By construction,
$$
\kappa(C\mathord{\cdot}z)=\sum_{x\in\mathcal{L}_j}|x_i|\mbox{.}
$$

Further denote $\mathcal{L}_1=C\mathord{\cdot}z$ and call $\mathcal{K}$ the union of $\mathcal{L}_1$ to $\mathcal{L}_{d/|C\mathord{\cdot}z|}$. By construction, $\mathcal{K}$, and its partition into the sets $\mathcal{L}_i$ have the desired properties. 
\end{proof}

We now prove Theorem \ref{PPP.sec.0.thm.5}. Recall that, according to this theorem, 
there is a unique subset of $\mathbb{P}^d_\circ$ that solves our primitive point packing problem if and only if $k$ coincides with $\kappa(B(d,p)\cap\mathbb{P}^d_\circ)$ for some integer $p$.

\begin{proof}[Proof of Theorem \ref{PPP.sec.0.thm.5}]
When $k$ coincides with $\kappa\!\left(B_1(d,p)\cap\mathbb{P}^d_\circ\right)$ for some integer $p$, the result is provided by Corollary 3.2 from \cite{DezaPourninSukegawa2020}. Assume that
$$
k\neq\kappa\!\left(B_1(d,p)\cap\mathbb{P}^d_\circ\right)
$$
for all $p$ and consider a  subsets $\mathcal{X}$ of $\mathbb{P}^d_\circ$ such that $|\mathcal{X}|=\delta_z(d,k)$ and $\kappa(\mathcal{X})\leq{k}$. Let us prove there is another such subset of $\mathbb{P}^d_\circ$.

If $\mathcal{X}$ is not invariant under the action of $C$ on $\mathbb{R}^d$, then there exists a integer $i$ such that $\tau^i(\mathcal{X})\neq\mathcal{X}$. However, $\tau^i(\mathcal{X})$ and $\mathcal{X}$ have the same cardinality and the same image by $\kappa$. Therefore, the lemma holds in this case, and we assume from now on that $\mathcal{X}$ is invariant under the action of $C$ on $\mathbb{R}^d$.

We review two cases. Assume first that $\mathcal{X}$ is equal to $B(d,p)\cap\mathbb{P}^d_\circ$ for some integer $p$. By assumption, $k$ must then be greater than $\kappa(\mathcal{X})$. Consider the point $a$ in $\mathcal{X}$ whose first coordinate is equal to $p-1$, whose second coordinate is equal to $1$, and whose all other coordinates, if any, are equal to $0$. Observe that adding $1$ to any of the coordinates of $a$ other than the second one results in a primitive point $b$ of $1$-norm $p+1$. In addition, since $\mathcal{X}$ is invariant under the action of $C$ on $\mathbb{R}^d$, the equality
$$
\sum_{x\in\mathcal{X}}|x_i|=\kappa(\mathcal{X})
$$
holds for every integer $i$ such that $1\leq{i}\leq{d}$. Hence,
$$
\mathcal{Y}=(\mathcal{X}\mathord{\setminus}\{a\})\cup\{b\}
$$
satisfies $\kappa(\mathcal{Y})=\kappa(\mathcal{X})+1$. Now recall that $k$ is greater than $\kappa(\mathcal{X})$. Therefore, $\mathcal{Y}$ is a subset of $\mathbb{P}^d_\circ$ distinct from $\mathcal{X}$ such that $|\mathcal{Y}|=\delta_z(d,k)$ and $\kappa(\mathcal{Y})\leq{k}$, as desired. Now assume that $\mathcal{X}$ is distinct from $B(d,p)\cap\mathbb{P}^d_\circ$ for every integer $p$. In this case, there exists a point $y$ in $\mathcal{X}$ and a point $z$ in $\mathbb{P}^d_\circ\mathord{\setminus}\mathcal{X}$ such that $\|z\|_1\leq\|y\|_1$. As $\mathcal{X}$ is invariant under the action of $C$ on $\mathbb{R}^d$, it must admit $C\mathord{\cdot}y$ as a subset and be disjoint from $C\mathord{\cdot}z$. Hence, if the orbits of $y$ and $z$ under the action of $C$ on $\mathbb{R}^d$ have the same cardinality, the set
$$
\mathcal{Y}=(\mathcal{X}\mathord{\setminus}C\mathord{\cdot}y)\cup{C\mathord{\cdot}z}
$$
is a subset of $\mathbb{P}^d_\circ$ distinct from $\mathcal{X}$ such that $|\mathcal{Y}|=\delta_z(d,k)$ and $\kappa(\mathcal{Y})\leq{k}$, as desired. Assume that $C\mathord{\cdot}y$ and $C\mathord{\cdot}z$ do not have the same cardinality.

According to Proposition \ref{PPP.sec.5.prop.1}, there exists a subset $\mathcal{K}_y$ of $d$ points from $\mathbb{P}^d_\circ$ of the same $1$-norm than $y$, and a partition $\mathcal{P}_y$ of $\mathcal{K}_y$ into $d/|C\mathord{\cdot}y|$ subsets, all of cardinality $C\mathord{\cdot}y$, such that any of these subsets $\mathcal{L}$ satisfies $\kappa(C\mathord{\cdot}y)=\kappa(\mathcal{L})$. In addition, $\mathcal{L}$ is contained in the orbit of some point from $\mathbb{P}^d_\circ$ under the action of $C$ on $\mathbb{R}^d$. Therefore, either $\mathcal{L}$ is a subset of $\mathcal{X}$, or it is disjoint from it. If any of the sets $\mathcal{L}$ in $\mathcal{P}_y$ is disjoint from $\mathcal{X}$, then the set
$$
\mathcal{Y}=(\mathcal{X}\mathord{\setminus}C\mathord{\cdot}y)\cup\mathcal{L}\mbox{.}
$$
is distinct from $\mathcal{X}$. Yet, it satisfies $|\mathcal{Y}|=\delta_z(d,k)$ and, as $\kappa(C\mathord{\cdot}y)$ coincides with $\kappa(\mathcal{L})$, then so do $\kappa(\mathcal{Y})$ and $\kappa(\mathcal{X})$. Therefore, in this case the lemma holds.

Now assume that $\mathcal{K}_y$ is a subset of $\mathcal{X}$. We proceed in the same way with $C\mathord{\cdot}z$. More precisely, by Proposition \ref{PPP.sec.5.prop.1}, there exists a set $\mathcal{K}_z$ made up of $d$ points from $\mathbb{P}^d_\circ$ of the same $1$-norm than $z$ and a partition of it into $d/|C\mathord{\cdot}z|$ subsets of cardinality $C\mathord{\cdot}z$ such that, if $\mathcal{L}$ is one of these subsets, then for every $i$,
\begin{equation}\label{PPP.sec.0.thm.5.eq.1}
\kappa(C\mathord{\cdot}z)=\sum_{x\in\mathcal{L}}|x_i|\mbox{.}
\end{equation}

Again, $\mathcal{L}$ is either a subset of $\mathcal{X}$ or disjoint from it. If any of the sets, say $\mathcal{L}$, in $\mathcal{K}_z$ is a subset of $\mathcal{X}$, then the set
$$
\mathcal{Y}=(\mathcal{X}\mathord{\setminus}\mathcal{L})\cup{C\mathord{\cdot}z}\mbox{.}
$$
is distinct from $\mathcal{X}$. Moreover by construction, it satisfies $|\mathcal{Y}|=\delta_z(d,k)$. Further note that, since, (\ref{PPP.sec.0.thm.5.eq.1}) holds for every $i$, $\kappa(\mathcal{Y})$ must be equal to $\kappa(\mathcal{X})$. As a consequence, the lemma also holds in this case.

Assume that $\mathcal{K}_z$ is disjoint from $\mathcal{X}$. In this case the set
$$
\mathcal{Y}=(\mathcal{X}\mathord{\setminus}\mathcal{K}_y)\cup\mathcal{K}_z
$$
has cardinality $\delta_z(d,k)$. Moreover, by construction, for every $i$,
$$
\frac{d\kappa(C\mathord{\cdot}y)}{|C\mathord{\cdot}y|}=\sum_{x\in\mathcal{K}_y}|x_i|\mbox{.}
$$

However, according to Proposition \ref{PPP.sec.3.prop.1}, the left-hand side of this equality is precisely the $1$-norm of $y$. Similarly, for every $i$,
$$
\|z\|_1=\sum_{x\in\mathcal{K}_z}|x_i|\mbox{.}
$$

Now recall that $\|z\|_1\leq\|y\|_1$. As a consequence, $\kappa(\mathcal{Y})\leq\kappa(\mathcal{X})$.
\end{proof}

\section{Asymptotic estimates}\label{PPP.sec.1.5}

In this section we provide sharp asymptotic estimates for some of the quantities we studied in the previous sections. We begin with the asymptotic behavior of $c_\psi(p,d)$ when $d$ is fixed and $p$ grows large. Recall that this quantity is the number of primitive points of $1$-norm $p$ contained in the positive orthant $]0,+\infty[^d$ or, in combinatorial terms, the number of compositions of $p$ into $d$ relatively prime integers. Let us first observe that there is no asymptotics in the $2$-dimensional case. Indeed, according to Theorem \ref{PPP.sec.1.thm.1}, when $p\geq2$,
$$
c_\psi(p,2)=J_1(p)\mbox{.}
$$

Recall that $J_1(p)$ is Euler's totient function. In particular, we have the following, where $\gamma$ stands for Euler's constant (see for instance \cite{HardyWright1938}).
\begin{prop}\label{PPP.sec.1.prop.2}
The following equalities hold.
\begin{enumerate}
\item[(i)] $\displaystyle\liminf_{p\to\infty}\frac{c_\psi(p,2)}{p}\log\log p=e^{-\gamma}$,
\item[(ii)] $\displaystyle\limsup_{p\to\infty}\frac{c_\psi(p,2)}{p}=1$.
\end{enumerate}
\end{prop}

Let us now consider the case of the dimensions above $2$. It is an immediate consequence of Theorem \ref{PPP.sec.0.thm.1} that
\begin{equation}\label{PPP.sec.1.5.eq.1}
\left|S(d,p)\cap\mathbb{P}^d_\circ\right|=\frac{1}{2}\sum_{j=1}^{d}2^j{d\choose{j}}c_\psi(p,j)\mbox{.}
\end{equation}

Moreover, it follows from Theorem 2.2 from \cite{DezaPourninSukegawa2020} that
\begin{equation}\label{PPP.sec.1.5.eq.2}
\lim_{p\to\infty}\frac{\left|B(d,p)\cap\mathbb{P}^d_\circ\right|}{p^d}=\frac{2^{d-1}}{d!\zeta(d)}\mbox{,}
\end{equation}
where $\zeta$ stands for Riemann's zeta function. These two equalities provide the asymptotic behavior of $c_\psi(p,d)$ when $d$ is fixed and $p$ grows large.

\begin{thm}\label{PPP.sec.1.5.thm.1}
When $d\geq3$, $\displaystyle\lim_{p\to\infty}\frac{c_\psi(p,d)}{p^{d-1}}=\frac{1}{(d-1)!\zeta(d)}$.
\end{thm}
\begin{proof}
Assume that $d\geq3$. According to (\ref{PPP.sec.1.5.eq.1}),
\begin{equation}\label{PPP.sec.1.5.thm.1.eq.1}
\frac{c_\psi(p,d)}{p^{d-1}}=\frac{\left|S(d,p)\cap\mathbb{P}^d_\circ\right|}{2^{d-1}p^{d-1}}-\sum_{j=1}^{d-1}\frac{1}{2^{d-j}}{d\choose{j}}\frac{c_\psi(p,j)}{p^{d-1}}\mbox{.}
\end{equation}

Moreover, it follows from (\ref{PPP.sec.1.5.eq.2}) that
\begin{equation}\label{PPP.sec.1.5.thm.1.eq.2}
\displaystyle\lim_{p\to\infty}\frac{\left|S(d,p)\cap\mathbb{P}^d_\circ\right|}{p^{d-1}}=\frac{2^{d-1}}{(d-1)!\zeta(d)}\mbox{.}
\end{equation}

Using (\ref{PPP.sec.1.5.thm.1.eq.1}) and (\ref{PPP.sec.1.5.thm.1.eq.2}), the theorem can be established by induction on $d$. Indeed, recall that $c_\psi(p,1)=0$ when $p\geq2$. In addition, by Proposition \ref{PPP.sec.1.prop.2},
$$
\displaystyle\lim_{p\to\infty}\frac{c_\psi(p,2)}{p^2}=0\mbox{.}
$$

Hence, combining (\ref{PPP.sec.1.5.thm.1.eq.1}) and (\ref{PPP.sec.1.5.thm.1.eq.2}) proves the theorem when $d$ is equal to $3$. Now assume that $d$ is greater than $3$ and that the theorem holds for all dimensions less than $d$ and greater than $2$. In this case, by induction,
$$
\displaystyle\lim_{p\to\infty}\frac{c_\psi(p,j)}{p^{d-1}}=0
$$
when $3\leq{j}<d$, and according to Proposition \ref{PPP.sec.1.prop.2},
$$
\displaystyle\lim_{p\to\infty}\frac{c_\psi(p,2)}{p^{d-1}}=0\mbox{.}
$$

As  $c_\psi(p,1)=0$ when $p$ is greater than $1$, we just need to combine (\ref{PPP.sec.1.5.thm.1.eq.1}) and (\ref{PPP.sec.1.5.thm.1.eq.2}) in order to get the desired asymptotic estimate.
\end{proof}

We also give the asymptotic behavior of $\left|B(d,p)\cap\mathbb{P}^d_\circ\right|$ and $\kappa\!\left(B(d,p)\cap\mathbb{P}^d_\circ\right)$ when $p$ is fixed and $d$ grows large. More precisely, we prove the following.

\begin{lem}\label{PPP.sec.1.5.lem.1}
$\displaystyle\lim_{d\to\infty}\frac{\left|B(d,p)\cap\mathbb{P}^d_\circ\right|}{d^p}=\frac{2^{p-1}}{p!}$ and $\displaystyle\lim_{d\to\infty}\frac{\kappa\!\left(B(d,p)\cap\mathbb{P}^d_\circ\right)}{d^{p-1}}=\frac{2^{p-1}}{(p-1)!}$.
\end{lem}
\begin{proof}
Observe that, when $d>p$, $c_\psi(p,d)=0$. In this case, (\ref{PPP.sec.1.5.eq.1}) yields
\begin{equation}\label{PPP.sec.1.5.lem.1.eq.1}
\frac{\left|B(d,p)\cap\mathbb{P}^d_\circ\right|}{d^p}=\frac{\left|B(d,p-1)\cap\mathbb{P}^d_\circ\right|}{d^p}+\frac{1}{2}\sum_{j=1}^p\frac{2^j}{d^p}{d\choose{j}}c_\psi(p,j)\mbox{.}
\end{equation}

The announced asymptotic behavior for $\left|B(d,p)\cap\mathbb{P}^d_\circ\right|$ can be obtained from this equality by induction on $p$. Indeed, as observed in \cite{DezaManoussakisOnn2018} (see theorem 3.2 therein), $\left|B(d,2)\cap\mathbb{P}^d_\circ\right|=d^2$, which provides the base case for the induction.

Now assume that $p$ is greater than $2$ and note that the limit
$$
\lim_{d\to\infty}\frac{1}{d^p}{d\choose j}
$$
is equal to $0$ when $j<p$ and to $1/p!$ when $j=p$. Moreover, by induction,
$$
\lim_{d\to\infty}\frac{\left|B(d,p-1)\cap\mathbb{P}^d_\circ\right|}{d^p}=0\mbox{.}
$$

As in addition, $c_\psi(p,p)=1$, letting $d$ grow large in (\ref{PPP.sec.1.5.lem.1.eq.1}) results in the desired asymptotic estimate for $\left|B(d,p)\cap\mathbb{P}^d_\circ\right|$ when $d$ is fixed and $p$ goes to infinity. Now recall that, as a consequence of Proposition \ref{PPP.sec.1.prop.1},
$$
\kappa\!\left(B(d,p)\cap\mathbb{P}^d_\circ\right)-\kappa\!\left(B(d,p-1)\cap\mathbb{P}^d_\circ\right)=\frac{p\!\left|S(d,p)\cap\mathbb{P}^d_\circ\right|}{d}\mbox{,}
$$

By another induction on $p$, this equality, together with our asymptotic estimate for $\left|B(d,p)\cap\mathbb{P}^d_\circ\right|$ provides the announced estimate for $\kappa\!\left(B(d,p)\cap\mathbb{P}^d_\circ\right)$. Note that the base case is given, again, by Theorem 3.2 from \cite{DezaManoussakisOnn2018}.
\end{proof}

The following estimate is an immediate consequence of Lemma \ref{PPP.sec.1.5.lem.1}.

\begin{thm}\label{PPP.sec.1.5.thm.2}
$\displaystyle\lim_{d\to\infty}\frac{1}{d}\frac{\left|B(d,p)\cap\mathbb{P}^d_\circ\right|}{\kappa\!\left(B(d,p)\cap\mathbb{P}^d_\circ\right)}=\frac{1}{p}$.
\end{thm}

\bibliography{PrimitivePointPacking}
\bibliographystyle{ijmart}

\end{document}